\documentclass[11pt,reqno]{amsart}

\usepackage{amsmath, wrapfig}
\usepackage{dsfont, a4wide, amsthm, amssymb, amsfonts, graphicx}
\usepackage{fancyhdr, xspace, psfrag, setspace, supertabular, color}
\usepackage[colorlinks]{hyperref}
\usepackage{bbm}
\usepackage{stmaryrd}
\usepackage{txfonts}
\usepackage{enumerate}
\usepackage[utf8]{inputenc}
\usepackage[T1]{fontenc}

\newtheorem{theorem}{Theorem}[section]

\newtheorem{lemma}[theorem]{Lemma}
\newtheorem*{claim1*}{Claim 1}
\newtheorem*{claim2*}{Claim 2}
\newtheorem*{definition*}{Definition}
\newtheorem{corollary}[theorem]{Corollary}

\newtheorem{definition}[theorem]{Definition}

\newtheorem{thm}{Theorem}[section]

\newtheorem{lem}[thm]{Lemma}

\newtheorem{cor}[thm]{Corollary}

\theoremstyle{definition}

\newtheorem{rk}[thm]{\bf Remark}

\DeclareMathOperator{\disc}{d}

\def\e{\epsilon}

\def\Z{\mathbb{Z}}
\def\R{\mathbb{R}}

\def\a{\alpha}
\def\be{\beta}

\def\d{\delta}

\def\b1{\mathbbm{1}}

\def\F#1{f(#1)}

\def\subset{\subseteq}
\def\supset{\supseteq}
\begin{document}

\title{Product-free subsets of $(0,1)$}

\author[L. Franchi]{Leonardo Franchi}
\address{Department of Pure Mathematics and Mathematical Statistics,
Centre for Mathematical Sciences, Wilberforce Road,
Cambridge CB3 0WA, United Kingdom}
\email{lf511@cam.ac.uk}

\author[W. T. Gowers]{W. T. Gowers}
\address{Coll\`ege de France, 11 place Marcelin Berthelot,
75005 Paris, France and Department of Pure Mathematics and Mathematical Statistics,
Centre for Mathematical Sciences, Wilberforce Road,
Cambridge CB3 0WA, United Kingdom} 
\email{W.T.Gowers@dpmms.cam.ac.uk}

\author[F. Yip]{Fredy Yip}
\address{Department of Pure Mathematics and Mathematical Statistics,
Centre for Mathematical Sciences, Wilberforce Road,
Cambridge CB3 0WA, United Kingdom}
\email{fy276@cam.ac.uk}

\begin{abstract}
The third problem in Ben Green's collection of 100 open problems asks whether an open subset of $(0,1)$ that does not contain $x,y,z$ with $xy=z$ must have measure at most 1/3. We give an affirmative answer to this question. As part of the proof we obtain a result of independent interest that gives a lower bound for the size of the sumset and the difference set of a set of reals in terms not just of its size but also of a parameter that measures how far it is from being an interval.
\end{abstract}

\maketitle

\section{Introduction}

In a remarkable paper from 2014 \cite{EberhardGreenManners2014}, Eberhard, Green and Manners proved that for every $\e>0$ there exists a set $A$ of $n$ integers such that no sum-free subset of $A$ has size more than $(1/3+\e)n$, answering a 1965 question of Erd\H os \cite{Erdos1965}, which was later asked again by several other authors. In the other direction, Erd\H os himself gave a short and elegant argument to show that there must be a sum-free subset of size at least $n/3$, which by being slightly more careful one can improve to $(n+1)/3$. This bound was improved by Bourgain \cite{Bourgain} to $(n+2)/3$, using a much more complicated argument, and in a recent breakthrough Bedert \cite{Bedert} obtained a bound of $n/3+c\log\log n$.

A question that arose naturally in connection with the work of Eberhard, Green and Manners, which was formally asked by Green in his collection of 100 open problems \cite{GreenOpenProblems}, was the following. Let $A$ be an open subset of the open interval $(0,1)$ and suppose that $A$ is \emph{product free}: that is, it is not possible to find $x,y,z\in A$ such that $xy=z$. Must $A$ have measure at most 1/3? An obvious source of open product-free sets is sets of the form $\bigcup_{n=-\infty}^\infty(\a^{3n+1},\a^{3n+2})$, where $\a$ is a constant greater than 1. Choosing $\a$ to be close to 1 and intersecting such a set with $(0,1)$, one obtains product-free open subsets of $(0,1)$ with measure arbitrarily close to 1/3, so an upper bound of 1/3 would, if true, be best possible. 

Heuristically, the reason the question is related is that if the answer is yes, then we can change variables by setting $y=-\log x$ (as we shall at the beginning of the next section), which converts the problem into showing that the exponential distribution on $\R$ has the property that every sum-free subset of $\R$ has measure at most $1/3$. Thus, a positive answer to Green's question can be seen as answering a continuous version of Erd\H os's question, and indeed continuous measures play an important role in \cite{EberhardGreenManners2014}.

The main result of this paper gives an affirmative answer to Green's question. We write $|A|$ for the Lebesgue measure of $A$.

\begin{theorem} \label{main1} Let $A$ be an open product-free subset of $(0,1)$. Then $|A|< 1/3$.
\end{theorem}

\noindent One of the main tools we use to prove this theorem is a second theorem, about the size of the sumset $A+A$ or the difference set $A-A$ when $A$ is a Borel measurable set of real numbers, the statement of which is itself closely related to various lemmas from the paper of Eberhard, Green and Manners. By Kneser's theorem \cite{Kneser1956} or the 1-dimensional case of the Brunn-Minkowski inequality (see for example \cite{Gardner2002}) we have that $|A-A|\geq 2|A|$. For compact sets, this bound is sharp if and only if $A$ is an interval, so one can hope to try to obtain a more precise result by quantifying appropriately the extent to which $A$ fails to be an interval. To this end, we define the \emph{positive interval discrepancy}, or just \emph{discrepancy}, of a set $A$ by the formula
\[d(A)=\sup_I \big(2|A\cap I|-|I|\big),\]
where the supremum is over all intervals $I$. Note that $2|A\cap I|-|I|=|A\cap I|-|I\setminus A|$, so this measures how much more $A$ can intersect an interval than $A^c$. Usually we shall consider compact sets $A$, so the supremum above will be a maximum. Indeed, nothing much is lost by taking $A$ to be a union of finitely many closed intervals, which we do for most of the paper, in which case the maximizing $I$ will always join the beginning of one of those intervals to the end of another. 

Our second theorem, which we shall in fact prove first, is the following.

\begin{theorem} \label{main2}
Let $A$ be a Borel measurable set of real numbers with finite measure. Then $|A+A|$ and $|A-A|$ are both at least $4|A|-2d(A)$.
\end{theorem}

To get a feel for this theorem, it is helpful to see that there is a large class of equality cases. If $A$ is itself an interval, then $d(A)=|A|$, since the maximizing interval $I$ is easily seen to be $A$ itself, and therefore $4|A|-2d(A)=2|A|=|A+A|$. 

Another important equality case is a set of the form
\[A=[0,t]\cup[1,1+t]\cup\dots\cup[n-1,n-1+t],\]
for some $t\in(0,1)$, which has measure $tn$. (More generally, one can of course take a set $P+[0,t]$ where $P$ is a finite arithmetic progression in $\R$ and $t$ is a positive real that is smaller than the common difference of $P$.) Then
\[A+A=[0,2t]\cup[1,1+2t]\cup\dots\cup[2(n-1),2(n-1)+2t].\]
If $t\leq 1/2$, then $[0,t]$ is a maximizing interval, so $d(A)=t$, and
\[|A+A|=2t(2n-1)=4tn-2t=4|A|-2d(A).\]
If $t>1/2$, then the unique maximizing interval is $[0,n-1+t]$, and $d(A)=2nt-(n-1+t)=n(2t-1)+1-t$. Therefore,
\[|A+A|=2(n-1)+2t=4nt-n(4t-2)-2(1-t)=4|A|-2d(A)\]
in this case as well.

This does not by any means exhaust the set of sharp examples. For instance, let $A=I\setminus E$, where $I$ is an interval. There are many sets $E$ that satisfy the following two conditions -- they just need to be fairly small, particularly near the edges of $I$. 
\begin{enumerate} 
\item The maximizing interval for $A$ is $I$.
\item $A+A=I+I$.
\end{enumerate}
If $E$ has these two properties, then $d(A)=2|A|-|I|$, so once again we have 
\[|A+A|=2|I|=4|A|-2d(A).\]

The proof we shall give of Theorem \ref{main2} is built up via a sequence of lemmas that provide a toolbox that is also used for proving Theorem \ref{main1}. The argument is elementary, in the sense that no use is made of tools such as Fourier analysis or regularity lemmas. 
\medskip

It is instructive to reformulate Theorem \ref{main2} in terms of the size of the symmetric difference of $A$ and an interval. If $A$ is a measurable set and $I$ is an interval, then 
\[2|A\cap I|-|I|=|A\cap I|-|I\setminus A|=|A|-|A\setminus I|-|I\setminus A|=|A|-|A\triangle I|.\]
Therefore, Theorem \ref{main2} is equivalent (if $A$ is compact) to the assertion that if $|A+A|\leq C|A|$ or $|A-A|\leq C|A|$, then there exists an interval $I$ such that 
\[2(|A|-|A\triangle I|)\geq (4-C)|A|,\]
or in other words that
\[|A\triangle I|\leq(C/2-1)|A|.\]
This is a non-trivial statement for every $C<4$. When $C=2$ it tells us that if $|A+A|\leq 2|A|$, then there is an interval $I$ with $|A\triangle I|=0$. This quickly implies that $A\subset I$, and therefore yields the familiar fact that $A$ is equal to $[\min A,\max A]$ up to a set of measure zero. When $C$ is close to 2, our result yields a stability version of this statement, which is also well known, see, for example, \cite[Corollary~3.1]{Ruzsa1991}.

For comparison, the regime $C<3$ lies in the sharp one-dimensional inverse theory for small sum-sets, and may be viewed as the symmetric, or diagonal, version of the stability theory for the one-dimensional Brunn-Minkowski inequality. Equivalently, it is the continuous counterpart of Freiman's $3k-4$ theorem \cite{Freiman1959}. A particularly convenient formulation of this diagonal statement is Theorem~1.1 of Figalli and Jerison \cite{FigalliJerison2015}, which asserts that for every Borel measurable set $A\subset\mathbb R$,
\[
    |A+A|-2|A|
    \geq
    \min\bigl\{
        |\operatorname{co}(A)\setminus A|,
        |A|
    \bigr\}.
\]
Equivalently, if $C<3$ it implies that $A$ is contained in an interval $J$ of length at most $(C-1)|A|$. 

This statement is already implicit in the earlier work of Ruzsa
\cite{Ruzsa1991}, whose exact two-set estimates contain the diagonal
result above as a special case. Figalli and Jerison point out that the
one-dimensional statement may also be deduced from Freiman's theorem,
and give a self-contained proof. The corresponding inverse theory for
two different sets was later developed further by de Roton
\cite{deRoton2018}, who proved a full continuous analogue of the
$3k-4$ theorem, together with additional structural information on the
sumset. The diagonal
inequality above also follows from Theorem \ref{main2}, by an
amplification argument. Since the proof is slightly less simple than the arguments given in this introduction, we defer it to the end of Section~3. 

If $C=4-\e$, we can conclude that there is an interval $I$ such that $2|A\cap I|-|I|\geq\frac{\e}{2}|A|$, which implies that $|I|\geq\frac{\e}{2}|A|$ and \[\frac{|A\cap I|}{|I|}\geq\frac 12+\frac{\e|A|}{4|I|}.\]
This is a sharp version of Theorem 6.2 of \cite{EberhardGreenManners2014}. If $A$ is the set $\{0,1,2,\dots,k-1\}+[0,1/2]$, then $|A|=k/2$ and $|A+A|=2k-1=(4-2/k)|A|$. If we take $I$ to be the interval $[0,(j-1)+1/2]$ for some $j<k$, then we have that 
\[\frac{|A\cap I|}{|I|}=\frac{j/2}{j-1/2}=\frac j{2j-1}=\frac 12+\frac 1{2(2j-1)}=\frac 12+\frac 1{4|I|}.\]
Since $\e|A|=1$ for this particular choice of $A$, this example demonstrates the sharpness (at least for an infinite sequence of $\e$ tending to zero). 
\medskip

The following two observations lead to further consequences of Theorem \ref{main2}. 

\begin{lemma} \label{discrepancy observation}
Let $A$ be a set of non-negative real numbers that is measurable  with finite measure, let $\d>0$, and suppose that $A\cap(A-\d)=\emptyset$. Then $d(A)\leq\d$.
\end{lemma}

\begin{proof}
Let $I=[a,b]$ and let $A'=A\cap I$. Then $A'$ and $A'-\d$ are disjoint sets of measure $|A\cap I|$ and $A'\cup(A'-\d)\subset[a-\d,b]$, so $2|A\cap I|\leq b-a+\d=|I|+\d$, from which it follows that $d(A)\leq\d$, as stated. 
\end{proof}
\noindent For later use we remark that if $A$ is closed, then the inequality is strict. Indeed, the interval $[a-\d,b]$ is connected and is not contained in either of the closed sets $A'$ or $A'-\d$, so it cannot be the union of those two sets. The complement of that union is therefore an open subset of $[a-\d,b]$, so it does not have full measure, so equality does not occur.

\begin{lemma} \label{discrepancy observation2}
Let $\d>0$ and let $A$ be a measurable set of real numbers with the property that $|A\cap I|\leq\d/2$ for every interval $I$ of length $\d$. Then $d(A)\leq\d$.
\end{lemma}

\begin{proof}
    Let $J$ be an arbitrary interval and let $K$ be a maximal subinterval of $J$ with length a multiple of $\d$. Then 
    \[|A\cap J|\leq|A\cap K|+|J\setminus K|\leq\frac 12|K|+|J\setminus K|.\]
    It follows that
    \[2|A\cap J|-|J|\leq|K|+2|J\setminus K|-|J|=|J\setminus K|\leq \d.\qedhere\]
\end{proof}

Combining Theorem \ref{main2} with Lemmas \ref{discrepancy observation} and \ref{discrepancy observation2} we obtain the following immediate consequence.

\begin{corollary} \label{consequences}
Let $A$ be a Borel measurable set of real numbers with finite measure, let $\d>0$, and suppose either that $A\cap(A-\d)=\emptyset$ or that $|A\cap I|\leq\d/2$ for every interval $I$ of length $\d$. Then $|A+A|$ and $|A-A|$ are both at least $4|A|-2\d$. \hfill $\square$
\end{corollary}

The contrapositive of Corollary \ref{consequences} is also worth noting: if $\min\{|A+A|,|A-A|\}<4|A|-2d$, then $d\in A-A$ and there is an interval $I$ of length $d$ with $|A\cap I|>d/2$.

One might hope to prove results such as Corollary \ref{consequences} by combining an inverse theorem with a contradiction argument. That is, one could assume that the sumset or difference set is small, use an inverse theorem to obtain some structure for $A$, and then try to show that this structure is incompatible with either of the hypotheses of the corollary. There are several powerful inverse results around the doubling threshold $4$, including those of Eberhard, Green and Manners \cite{EberhardGreenManners2014}, van Hintum and Keevash \cite{vanHintumKeevash2026}, and Jing and Mudgal \cite{JingMudgal2026}. However, the structural conclusions of these results do not seem to be sufficiently precise to obtain anything close to the bounds above.

\medskip

Theorem \ref{main2} also has a slightly curious consequence for sumsets of sets of integers, which does not seem to be easy to prove directly. 

\begin{corollary}\label{discrete} Let $k$ be a positive integer and let $A\subset\Z$ be a finite set such that $A\cap(A-k)=\emptyset$. Then  \[|A+A+\{0,1\}|\geq 4|A|-2k.\] 
\end{corollary}

\begin{proof} Applying Theorem \ref{main2} to the set $A'=A+[0,1]$, we obtain that $|A'+A'|\geq 4|A'|-2d(A')$. But $|A'\cap(A'-k)|=0$, from which it follows that $d(A')\leq k$, by Lemma \ref{discrepancy observation} applied to the set $A+[0,1)$). Thus, $|A'+A'|\geq 4|A|-2k$. But $A'+A'=A+A+[0,2]=A+A+\{0,1\}+[0,1]$, and the result follows. 
\end{proof}

If $A=\{0,2k,\dots,2(n-1)k\}+\{0,1,\dots,k-1\}$, then $|A|=nk$ and 
\[A+A+\{0,1\}=\{0,1,\dots,4(n-1)k+2k-1\},\] 
which has size $(4n-2)k=4nk-2k$, so Corollary \ref{discrete} is sharp for this class of examples. 

One can in fact show that Corollary \ref{discrete} is equivalent to Corollary \ref{consequences} with the first condition. The proof of Corollary \ref{discrete} shows this in one direction. To prove the reverse implication, one approximates a measurable set $A'$ by a suitably scaled set of the form $A+[0,1]$ with $A\subset\Z$.

\subsection{Statement about AI use}

The entirety of the proof in this paper was discovered and written up by the authors. Though AI did not contribute any of the ideas, at certain stages of the research process it saved us time by finding counterexamples to statements that might well have been helpful to us had they been true. 

One such counterexample arose from the following observation. Let $I_1,I_2,\dots,I_k$ be a sequence of intervals with $I_r=[a_r,b_r]$ for each $r$, satisfying the following two properties.
\begin{enumerate}
    \item $2b_r\leq a_r+a_{r+1}$ and $b_r+b_{r+1}\leq 2a_{r+1}$ for every $r$.
    \item For every $r\leq s-2$, $I_r+I_s\subset I_{r+1}+I_{s-1}$.
\end{enumerate}
The second condition can be thought of as a kind of convexity property of the set $\bigcup_{r=1}^k\{r\}\times I_r$. If we set $A$ to be $\bigcup_{r=1}^kI_r$, then we have that $|A|=\sum_{r=1}^k|I_r|$ and 
\[|A+A|=2\sum_{r=1}^k|I_r|+\sum_{r=1}^{k-1}(|I_r|+|I_{r+1}|)=4|A|-|I_1|-|I_k|.\]
In the light of this example, it is tempting to conjecture that what matters is not so much what happens to $2|A\cap I|-|I|$ for an arbitrary interval $I$, but what happens for intervals at one or the other end of $A$. More precisely, if we define $d_L(A)$ to be the maximum of $2|A\cap I|-|I|$ over all intervals $[\min A,x]$ and $d_R(A)$ to be the maximum over all intervals $[x,\max A]$, one might conjecture that $|A+A|$ is always at least $4|A|-d_L(A)-d_R(A)$. However, ChatGPT 5.4 Pro provided for us the set $A=[0,1]\cup[4,6]\cup[8,11]$, which satisfies the second condition but not the first, and can easily be checked not to satisfy the inequality. 

\subsection*{Acknowledgements}
L.F. acknowledges support from the Isaac Newton Trust through a Trinity Cambridge Research Studentship. 

\section{From product sets to sumsets}

Let $E\subset(0,1)$ be an open set and let $A=-\log E$ (that is, $A=\{-\log x:x\in E\}$). The condition that $E$ is product free is equivalent to the statement that $A$ is sum free. Writing $\mathbbm 1_X$ for the characteristic function of $X$, we also have, using the substitution $u=-\log x$, that 
\[|E|=\int_0^1\mathbbm 1_E(x)\,\mathrm dx=\int_0^\infty\mathrm e^{-u}\mathbbm 1_A(u)\,\mathrm du.\]
Thus, the question we are addressing is equivalent to asking how large $\int_0^\infty e^{-u}\mathbbm 1_A(u)\,\mathrm du$ can be if $A$ is sum free. 

It will be convenient to work instead with closed sets: since open sets can be approximated by closed sets, any upper bound we find for closed sets implies the same upper bound for open sets. Assume, then, that $A$ is closed, which implies that it has a non-zero minimal element $\d$. Since $A$ is sum free, we have the condition that $A\cap(A-\d)=\emptyset$. Lemma \ref{discrepancy observation} then shows that we have a trade-off: in order to maximize $\int_0^\infty \mathrm e^{-u}\mathbbm 1_A(u)\,\mathrm du$ we would like $\min A$ to be as small as possible, but if $\min A$ is small, then we pay for that with the sumset and difference set of $A$ being large, which, since $A$ is sum free, confines $A$ to a smaller set. 

We write $\F{x}$ for $|A\cap[0,x]|=\int_0^x\mathbbm 1_A(u)\,\mathrm du$. Then integrating by parts gives
\[\int_0^\infty\mathrm e^{-u}\mathbbm 1_A(u)\,\mathrm du=\Big[\mathrm e^{-u}\F{u}\Big]_0^\infty+\int_0^\infty\mathrm e^{-u}\F{u}\,\mathrm du=\int_0^\infty\mathrm e^{-u}\F{u}\,\mathrm du.\]
From the fact that $A$ is sum free, it follows that $\F{u}\leq u/2$, so we obtain immediately the upper bound
\[\int_0^\infty\mathrm e^{-u}\F{u}\,\mathrm du\leq\frac 12\int_0^\infty u\mathrm e^{-u}\,\mathrm du=\frac 12.\]

With the help of Theorem \ref{main2} and Lemma \ref{discrepancy observation}, we can improve on this. Let $A_u=A\cap[0,u]$. Since $A$ is disjoint from $A-A$ and $A_u-A_u$ is a symmetric set contained in $[-u,u]$, then from the fact that $|A_u-A_u|\geq 4|A_u|-2d(A_u)$ it follows that $|(A_u-A_u)\cap[0,u]|\geq 2|A_u|-d(A_u)$ and hence that $|A_u|\leq u-2|A_u|+d(A_u)$, or in other words that $\F{u}\leq(u+d(A_u))/3$. This gives us an upper bound of
\[\frac 13+\frac 13\int_0^\infty\mathrm e^{-u}d(A_u)\,\mathrm du.\]
If $\min A=\d$, then by Lemma \ref{discrepancy observation} we have that $d(A_u)\leq\d$ for every $u$. But we also know that $\F{u}=0$ when $u\leq\d$ and $\F{u}\leq u-\d$ when $\d\leq u\leq 2\d$. When $2\d\leq u\leq 3\d$, we also have that $A\cap[0,u]\subset A\cap[\d,3\d]$, and since $A$ and $A+\d$ are disjoint, we have that $|A\cap[\d,3\d]|\leq\d$, so $\F{u}\leq\d$ for $u\in[2\d,3\d]$. Incorporating these observations and optimizing $\d$ (at approximately $0.251537$) one can obtain an upper bound of approximately $0.36909$. With a couple more observations of a similar kind and a bit of effort one can push the bound down to a little over $0.3415$, but it does not seem possible to reach $1/3$ without a further idea.

The further idea we use, which will occupy the bulk of this paper, is a proof of the following result.

\begin{lemma} \label{key lemma} Let $A$ be a measurable sum-free subset of $\R$ and let $\min A = 1$. For each $u\geq 0$ let $\F{u}=|A\cap[0,u]|$. Then for every $u$ we have the inequality $\F{u-1}+\F{u}+\F{u+1}\leq u$. 
\end{lemma}

Note that if $A$ is the set $[1,2)\cup[4,5)\cup[7,8)\cup\dots$, then $\frac{\mathrm d}{\mathrm du}\big(\F{u-1}+\F{u}+\F{u+1}\big)=1$ for almost every $u\geq 0$, since exactly one of $u-1,u$ and $u+1$ belongs to $A$. Since $\F{-1}+\F{0}+\F{1}=0$, it follows that $\F{u-1}+\F{u}+\F{u+1}=u$ for every $u$. Thus, the lemma is sharp for this example. Note also that the lemma implies after a simple rescaling that if $\min A=\d$, then $\F{u-\d}+\F{u}+\F{u+\d}\leq u$. 

We now show that Lemma \ref{key lemma} implies Theorem \ref{main1}. 

\begin{proof}[Reduction of Theorem \ref{main1} to Lemma \ref{key lemma}]

We have seen that it is enough to prove that if $A$ is a closed sum-free subset of the positive reals, then $\int_0^\infty\mathrm e^{-u}\F{u}\,\mathrm du\leq 1/3$. Let $\d=\min A$. Then using Lemma \ref{key lemma} we have that
\begin{align*}1&=\int_0^\infty u\mathrm e^{-u}\,\mathrm du\\
&\geq\int_0^\infty\mathrm e^{-u}(\F{u-\d}+\F{u}+\F{u+\d})\,\mathrm du\\
&=\int_0^\infty \F{u}(\mathrm e^{-(u+\d)}+\mathrm e^{-u}+\mathrm e^{-(u-\d)})\,\mathrm du\\
&=(\mathrm e^{-\d}+1+\mathrm e^\d)\int_0^\infty\mathrm e^{-u}\F{u}\,\mathrm du.\\
\end{align*}
The second equality in the above calculation made use of the fact that $\F{u}=0$ for $u\leq\d$.

Since $\cosh(\d)\geq 1$ for all $\d$, it follows that $\int_0^\infty \mathrm e^{-u}\F{u}\,\mathrm du\leq 1/3$, as desired. (In fact, since $\d>0$, we have a strict inequality.) 
\end{proof}

At this point it is less obvious what the role of Theorem \ref{main2} is. The answer is that it will turn out to be very useful in the proof of Lemma \ref{key lemma}, as will parts of its proof. But for now, let us observe that it implies an inequality that is similar to but weaker than Lemma \ref{key lemma}.

\begin{lemma} \label{weaker inequality}
Suppose that Theorem \ref{main2} is true and let $A$ be a Borel measurable subset of $\R$ such that $\min A=1$ and $A$ is sum free. Let $\F{u}=|A\cap[0,u]|$ for every non-negative $u\in\R$. Then 
\[\F{u-1}+2\F{u}\leq u\] 
for every $u$. 
\end{lemma}

\begin{proof} The positive part of $A_u-A_u$ is contained in $[0,u-1]$ and is disjoint from $A_{u-1}$. By Theorem \ref{main2} it also has measure at least $2|A_u|-d(A_u)$. But by Lemma \ref{discrepancy observation} we have that $d(A_u)\leq 1$, and putting all that together we deduce that $2|A_u|+|A_{u-1}|\leq u-1+1=u$, as claimed.
\end{proof}

\section{Proof of Theorem \ref{main2}}

Throughout the rest of the paper, unless we state otherwise we shall assume that sets named $A$ and $B$ are unions of finitely many finite closed intervals. Given such a set $A$ and an interval $I$, we define $d_A(I)$ to be $2|A\cap I|-|I|$, so the discrepancy $d(A)$ is by definition $\max_I d_A(I)$. 

Our eventual aim will be to prove the result by induction on the number of intervals that make up $A$. To that end, it will be useful to have conditions under which we can guarantee that $(A\cap I)+(A\cap J)=I+J$, since that raises the possibility of filling in gaps between the intervals that make up $A$. Before we start on this, and draw various consequences from it, we give a couple of motivating examples. 

We have already commented that if $I$ is an interval and we let $A=I\setminus E$ for a suitable small set $E$, then $A+A=I+I$. The word ``suitable" is important, however: for example, if $A=[0,a]\cup[b,1]$, then 
\[A+A=[0,2a]\cup[b,1+a]\cup[2b,2],\]
which equals $[0,2]$ if $b\leq 2a$ and $2b\leq 1+a$, or equivalently $1-a\leq 2(1-b)$. Note that the first condition is precisely what is needed to guarantee that $d_A([0,x])\geq 0$ for every $x\in[0,1]$ and the second is what is needed to guarantee that $d_A([x,1])\geq 0$ for every $x\in[0,1]$. 

Another example that explains some of the conditions we impose in what follows is the simple observation that we cannot hope for $A+B$ to equal $[\min A + \min B,\max A+\max B]$ if one of the sets is too small compared with the other. Consider for example the sets $A=[0,1]\cup[2,3]$ and $B=[0,4]\cup[8,12]$. They both satisfy the conditions just discussed, but $A+B=[0,7]\cup[8,14]$. In this case the diameter of $A$ is smaller than the length of the gap inside $B$, so there is no way that $A$ can close the gap in $B$. 

\subsection{Balanced intervals and their properties}

In this subsection we develop tools that will be put to use to prove Theorem \ref{main2}. 

\begin{definition} Let $I=[a,b]$ be an interval. A \emph{left subinterval} of $I$ is a subinterval of the form $[a,c]$ for some $c\in I$ and a \emph{right subinterval} is a subinterval of the form $[c,b]$. 
\end{definition}

\begin{definition} Let $A$ be a finite union of finite closed intervals. An interval $I$ is \emph{left balanced} with respect to $A$ if $d_A(J)\geq 0$ for every left subinterval $J$ of $I$. It is \emph{right balanced} if $d_A(J)\geq 0$ for every right subinterval $J$ of $I$. If $I$ is both left balanced and right balanced with respect to $A$, then it is \emph{balanced}. 
\end{definition} 

If $I$ is written as a disjoint union $I_1\cup\dots\cup I_k$ of subintervals, then $d_A(I)=\sum_{j=1}^kd_A(I_j)$. It follows that $I$ is balanced if and only if $d_A(J)\leq d_A(I)$ for every subinterval $J$ of $I$. From that it follows in turn that if $I$ is balanced, then $d_A(I)\geq 0$. Observe also that if $I=[a,b]$ is left balanced with respect to $A$ and $a<b$, then $a$ must belong to $A$ and not be the right end-point of one of the component intervals of $A$, since otherwise there is some left subinterval $[a,a+\e]$ such that $(a,a+\e]$ is disjoint from $A$, and for that subinterval we have $d_A([a,a+\e])=-\e<0$. Thus, if $I$ is a balanced interval of positive length, then $d_A(I)$ is strictly positive.

If $I$ is an interval $[a,a]$ of length 0, we shall adopt the convention that $I$ is left balanced, right balanced or balanced with respect to $A$ if and only if $a\in A$. This convention will not play an important role in the proof, but will allow us to make certain statements without having to stipulate that the intervals we are talking about have positive length. 

\begin{lemma} \label{union left balanced} Let $I$ and $J$ be intervals that are left balanced with respect to $A$. If $I\cap J\ne\emptyset$, then $I\cup J$ is also left balanced.
\end{lemma}

\begin{proof}
If $I\subset J$ or $J\subset I$ then we are done immediately. Otherwise, we may assume without loss of generality that $\min I\leq\min J$, which implies that $\min I<\min J$ and $\max I<\max J$. Let us therefore set $I=[a,b]$ and $J=[c,d]$ with $a<c\leq b<d$, which gives that $I\cup J=[a,d]$. 

Now let $x\in[a,d]$. Then $x\leq b$ or $x\geq c$. If $x\leq b$, then $d_A([a,x])\geq 0$ by the fact that $I$ is left balanced. If $x\geq c$, then 
\[d_A([a,x])=d_A([a,c])+d_A([c,x])\geq d_A([a,c])\geq 0,\]
where the first inequality uses the fact that $J$ is left balanced and the second the fact that $I$ is left balanced.
\end{proof}

This result has some immediate consequences, which we collect together into a single corollary.

\begin{corollary} \label{balanced intervals} The following statements hold with respect to $A$ and $B$.
\begin{enumerate}[(i)]
\item Any two distinct maximal left-balanced intervals are disjoint.
\item If $I$ and $J$ are balanced intervals and $I\cap J\ne\emptyset$, then $I\cup J$ is balanced.
\item Any two distinct maximal balanced intervals are disjoint.
\end{enumerate}
\end{corollary}

\begin{proof} If $I$ and $J$ are distinct maximal left-balanced intervals, then $I$ is not a subset of $J$ and $J$ is not a subset of $I$. Therefore, by maximality, $I\cup J$ is not left balanced, which by Lemma \ref{union left balanced} implies that $I$ and $J$ are disjoint. By symmetry, Lemma \ref{union left balanced} also holds with ``left balanced" replaced by ``right balanced", and therefore it also holds with ``left balanced" replaced by ``balanced". Therefore,   a similar argument tells us that two distinct maximal right-balanced intervals are disjoint, and that two distinct maximal balanced intervals are disjoint. 
\end{proof}

\begin{lemma} \label{intersect balanced}
Let $I$ be balanced with respect to $A$ and let $J$ be balanced with respect to $B$. If $I\cap J\ne\emptyset$ and $d_A(I)=d_B(J)$, then $(A\cap I)\cap (B\cap J)\ne\emptyset$.
\end{lemma}

\begin{proof}
If $d_A(I)=d_B(J)=0$, then $I$ and $J$ are singletons, so $I=J$ and by convention the singleton point belongs both to $A$ and to $B$, so the result follows. We may therefore assume that $d_A(I)=d_B(J)>0$ and therefore that $I$ and $J$ have positive length.

Let $I=[a,b]$ and let $J=[c,d]$. Without loss of generality $a\leq c$, which implies that $c\leq b$. If $b\leq d$, then $[c,b]$ is a subinterval of both $I$ and $J$, which are both balanced, so $d_A([c,b])+d_B([c,b])\geq 0$. If $c=b$, then since balanced intervals contain their end points, we have that $c\in(A\cap I)\cap (B\cap J)$. If $c<b$, then both $A\cap[c,b]$ and $B\cap[c,b]$ intersect $[c,d]$ in sets of measure at least $(b-c)/2$, which implies, since $A$ and $B$ are closed, that $A\cap B\cap[c,b]$ is non-empty. Since $[c,b]=I\cap J$, that gives us the conclusion in this case as well.

If $b>d$, then $J\subset I$, and 
\[d_A(J)=d_A([a,d])+d_A([c,b])-d_A(I)\geq -d_A(I)=-d_B(J).\]
It follows that $d_A(J)+d_B(J)\geq 0$. Since $J$ has positive length, this implies that $A\cap B\cap J$ is non-empty, and since $J\subset I$ we are done.
\end{proof}

\begin{corollary} \label{add balanced} Let $I$ be balanced with respect to $A$ and let $J$ be balanced with respect to $B$. If $d_A(I)=d_B(J)$, then $(A\cap I)+(B\cap J)=I+J$.
\end{corollary}

\begin{proof}
Let $x\in I+J$ and note that $x-J$ is balanced with respect to $x-B$ and that $d_{x-B}(x-J)=d_B(J)=d_A(I)$. It follows from Lemma \ref{intersect balanced} that $(A\cap I)\cap((x-B)\cap(x-J))\ne\emptyset$, which implies that $x\in(A\cap I)+(B\cap J)$. Therefore,   $I+J\subset(A\cap I)+(B\cap J)$. The other containment is trivial.
\end{proof}

By replacing $B$ and $J$ with $-B$ and $-J$, we can deduce from Corollary \ref{add balanced} the corresponding result for difference sets. 

\begin{lemma} \label{partition} Let $I$ be a maximal left-balanced interval with respect to $A$. Then $A\cap I$ can be partitioned into sets $A\cap I_1,\dots,A\cap I_k$, where $I_1<I_2<\dots<I_k$ are maximal balanced intervals with respect to $A$. Moreover,
\[d_A(I_i)\leq d_A([\min I_1,\max I_i])\leq d_A(I_1)\]
for every $i$.
\end{lemma}

\begin{proof}
Let $A\cap I=[x_1,y_1]\cup\dots\cup [x_n,y_n]$. Then every maximal balanced interval must be of the form $[x_i,y_j]$ for some $j\geq i$. Indeed, if its left end point does not belong to $[x_i,y_i)$ for some $i$, then it is not balanced, and if it belongs to $(x_i,y_i)$ and is balanced, then it is not a maximal balanced interval, and similarly for the right end point. 

Since each interval $[x_i,y_i]$ is contained in a maximal balanced interval and any two distinct maximal balanced intervals are disjoint, the intersections of the maximal balanced intervals with $A$ form a partition of $A$. By Lemma \ref{union left balanced}, if a balanced interval $J$ intersects $I$, then $I\cup J$ is left balanced, and therefore $J\subset I$, by the maximality of $I$. It follows that the intersections of the maximal balanced intervals with $A\cap I$ form a partition of $A\cap I$.

Let $a=x_1, b=y_n$ and let $I_i=[a_i,b_i]$ for each $i$. Then $I=[a,b']$ for some $b'\geq b$, and since it is left balanced, we have that $d_A([a,a_i])\geq 0$ for every $i$, from which it follows that 
\[d_A(I_i)\leq d_A([a,a_i])+d_A([a_i,b_i])=d_A([\min I_1,\max I_i]).\]

If the second inequality is not true, let $i$ be minimal such that $d_A([a,b_i])>d_A(I_1)$. Clearly $i\geq 2$. Since $[a,b_i]$ is left balanced but not balanced, it is not right balanced, so there exists $x\in [a,b_i]$ such that $d_A([x,b_i])<0$. Since $I_i$ is balanced, we must have that $x<a_i$. 

If $x\in(b_j,a_{j+1})$ for some $j$, then $d_A([b_j,b_i])<d_A([x,b_i])$, so we may assume that $x\in[a_j,b_j]$ for some $j$. Since $I_j$ is right balanced, we also have that
\[d_A([b_j,b_i])=d_A([x,b_i])-d_A([x,b_j])\leq d_A([x,b_i]),\]
so we may in fact assume that $x=b_j$. But then
\[d_A([a,b_j])=d_A([a,b_i])-d_A([b_j,b_i])>d_A([a,b_i])>d_A(I_1),\]
which contradicts the minimality of $i$. 
\end{proof}

\begin{lemma} \label{main step} Let $I=[a,b]$ and let $J=[c,d]$. Suppose that $I$ is right balanced with respect to $A$, $J$ is left balanced with respect to $B$, and $d_A(I)\geq d(B)$. Then
\[(B\cap J)-(A\cap I)\supset[c-b,c-b+2|B\cap J|].\]
\end{lemma}

\begin{proof}
Let $x\in[c-b,d-b]$ and suppose that $x\leq c-a$, so $a\leq c-x$. Since $J$ is left balanced with respect to $B$, $d_B([c,b+x])\geq 0$, and since $I$ is right balanced with respect to $A$, $d_A([c-x,b])\geq 0$. But $d_A([c-x,b])=d_{A+x}([c,b+x])$, so this implies that 
\[(A+x)\cap B\cap[c,b+x]\ne\emptyset.\] 
But $[c-x,b]\subset[a,b]$ and $[c,b+x]\subset[c,d]$, so this shows that $x\in(B\cap J)-(A\cap I)$. 

Now suppose that $x\in[c-b,d-b]$ and $x>c-a$. Then $[a+x,b+x]\subset [c,d]=J$, so by the definition of discrepancy we have that $d_B([c,a+x])\leq d(B)$, which by hypothesis is at most $d_A(I)$. It follows that
\[d_B([a+x,b+x])+d_A(I)\geq d_B([a+x,b+x])+d_B([c,a+x])=d_B([c,b+x])\geq 0,\]
since $J$ is left balanced with respect to $B$. But $d_A(I)=d_{A+x}([a+x,b+x])$, so it follows that $(A+x)\cap B\cap[a+x,b+x]\ne\emptyset$. Since $I+x=[a+x,b+x]\subset J$, it follows that $x\in(B\cap J)-(A\cap I)$ in this case too.

Now suppose that $x\in[d-b,c-b+2|J\cap B|]$. Then 
\begin{equation}
d_B(J)=2|J\cap B|-(d-c)=c-b+2|J\cap B|-(d-b)\geq x-d+b.\label{first}
\end{equation}
Also,
\begin{equation}
d_A([a,d-x])=d_A(I)-d_A([d-x,b])\geq d_A(I)-(b-d+x).\label{second}
\end{equation}

If $x<c-a$ and hence $c-x>a$, then
\[d_A([c-x,d-x])=d_A([c-x,b])-d_A([d-x,b])\geq d-x-b,\]
since $I$ is right balanced and $d_A([d-x,b])\leq b-(d-x)$. It follows that
\[d_{A+x}([c,d])+d_B([c,d])\geq 0,\]
and hence that $(A+x)\cap B\cap [c,d]\ne\emptyset$. Since also $[c-x,d-x]\subset[a,b]$, this implies that $x\in(B\cap J)-(A\cap I)$.

If $x\geq c-a$, then by \eqref{first} and our hypothesis, we have
\[d_B(J)\geq x-d+b+d(B)-d_A(I).\]
By \eqref{second} it follows that
\[d_{A+x}([a+x,d])+d_B([c,d])=d_A([a,d-x])+d_B(J)\geq d(B)\geq d_B([c,a+x]),\]
where the last inequality follows from the definition of $d(B)$ and the fact that $a+x\geq c$. It follows that
\[d_{A+x}([a+x,d])+d_B([a+x,d])\geq 0,\]
and therefore that $(A+x)\cap B\cap[a+x,d]\ne\emptyset$.
Since $a+x\geq c$ and $d-x\leq b$, it follows once again that $x\in (B\cap J)-(A\cap I)$.
\end{proof}

\begin{corollary} \label{main step sums}
Let $I=[a,b]$ and let $J=[c,d]$. Suppose that $I$ is left balanced with respect to $A$, $J$ is left balanced with respect to $B$, and $d_A(I)\geq d(B)$. Then 
\[(A\cap I)+(B\cap J)\supset[a+c,a+c+2|B\cap J|].\]
If $I$ is right balanced with respect to $A$, $J$ is right balanced with respect to $B$, and $d_A(I)\geq d(B)$, then
\[(A\cap I)+(B\cap J)\supset[b+d-2|J\cap B|,b+d].\]
\end{corollary}

\begin{proof}
The first assertion follows from Lemma \ref{main step} applied to $-A$ and $-I$ instead of $A$ and $I$. The second follows from the first if we replace all sets by their reflections in the origin.
\end{proof}

For the next corollary, we introduce an additional piece of notation.

\begin{definition} \label{edge discrepancy}
Let $A$ be a compact subset of $\R$. Then the \emph{right discrepancy} of $A$, denoted $d_R(A)$ is the maximum of $d_A([x,\max A])$ over all $x\leq\max A$, and the \emph{left discrepancy} of $A$, denoted $d_L(A)$, is the maximum of $d_A([\min A,x])$ over all $x\geq\min A$.
\end{definition}

Note that $d_R(A)$ and $d_L(A)$ are both at most $d(A)$.

\begin{corollary} \label{key bound}
Let $I$ be left balanced with respect to $A$ and suppose that $d_A(I)\geq d(B)$. Then 
\[\big|(B+(A\cap I))\cap[\min(A\cap I)+\min B,\min(A\cap I)+\max B+d_R(B)]\big|\geq 2|B|.\]
If instead $I$ is right balanced with respect to $A$, then
\[\big|(B+(A\cap I))\cap[\max(A\cap I)+\min B-d_L(B),\max(A\cap I)+\max B]\big|\geq 2|B|.\] 
\end{corollary}

\begin{proof}
Let $I=[a,b]$ and let $[c,d]=[\min B,\max B]$. Let $B\cap J_1<\dots<B\cap J_k$ be the partition of $B$ given by the maximal left-balanced intervals $J_i$, which are disjoint by Corollary \ref{balanced intervals}. Let $J_i=[c_i,d_i]$. Thus, our task is to prove that 
\[|((A\cap I)+B)\cap[a+c,a+d+d_R(B)]|\geq 2|B|.\]

For each $i$, the interval $J_i$ is left balanced with respect to $B$, and by hypothesis $d_A(I)\geq d(B)$. Therefore, since $I$ is left balanced with respect to $A$, Corollary \ref{main step sums} gives us that
\[
(A\cap I)+(B\cap J_i)\supset[a+c_i,a+c_i+2|B\cap J_i|].
\]
Let $K_i=[a+c_i,a+c_i+2|B\cap J_i|]$ for each $i$. We show first that the sets $K_i$ are disjoint. For this we need the inequality $c_{i+1}>c_i+2|B\cap J_i|$. But if that were not the case, then $2|B\cap J_i|\geq c_{i+1}-c_i$, and therefore for every $x\in[d_i,c_{i+1}]$ we would have
\[d_B([c_i,x])=2|[c_i,x]\cap B|-|[c_i,x]|\geq 2|B\cap J_i|-(c_{i+1}-c_i)\geq 0.\]
Since $J_i$ and $J_{i+1}$ are both left balanced, this would imply that $[c_i,d_{i+1}]$ is left balanced, contradicting the maximality of $J_i$ and $J_{i+1}$. 

Since $K_1<\dots<K_k$ and $\sum_i|K_i|=2\sum_i|B\cap J_i|=2|B|$, it suffices at this point to prove that 
\[\max K_k=a+c_k+2|B\cap J_k|\leq a+d+d_R(B).\]
But this is equivalent to the assertion that $d_R(B)\geq 2|B\cap J_k|-|[c_k,d]|$, which equals $d_B([c_k,d])$, since $\max(B\cap J_k)=d$. Therefore, this follows from the definition of $d_R(B)$. 

The second assertion follows from the first one applied to $-I$, $-A$ and $-B$, together with the observations that in this case $-I$ is left balanced with respect to $-A$, $d_R(-B)=d_L(B)$, and therefore that
\begin{align*}-[\min(-A\cap -I)+\min(-B),&\min(-A\cap -I)+\max(-B)+d_R(-B)]\\
&=[\max(A\cap I)+\min B-d_L(B),\max(A\cap I)+\max B].\qedhere\\
\end{align*}
\end{proof}

\subsection{Proof of Theorem \ref{main2}}

We shall actually prove the following slightly more general theorem, which has the two conclusions of Theorem \ref{main2} as special cases.

\begin{theorem}\label{offdiagonal main2}
Let $A$ and $B$ be (non-empty) unions of finitely many finite closed intervals, and suppose that $d(A)=d(B)$. Then
\[|A+B|\geq 2|A|+2|B|-d(A)-d(B).\]
\end{theorem}

\begin{proof} Let $n(A)$ and $n(B)$ be the number of intervals that make up $A$ and $B$, respectively. We shall apply induction on $n(A)+n(B)$. When $n(A)+n(B)=2$, so $n(A)=n(B)=1$, we have that $A$ and $B$ are intervals, so $|A+B|=|A|+|B|$, $d(A)=|A|$, and $d(B)=|B|$, so we have equality.

We begin by partitioning $A$ into sets $A\cap I_1 < \dots < A\cap I_k$ and $B$ into sets $B\cap J_1<\dots<B\cap J_\ell$, where the $I_i$ and $J_i$ are maximal balanced intervals in $A$ and $B$, respectively.

We now distinguish two cases. The first is when $d_A(I_1)=d_A(I_k)=d_B(J_1)=d_B(J_\ell)=d(A)$. If this does not hold, then one of the four quantities is strictly less than $d(A)$ (and therefore, by hypothesis, $d(B)$ as well). Since we can replace $(A,B)$ by $(B,A)$, $(-A,-B)$ or $(-B,-A)$ we may assume without loss of generality that $d_B(J_1)\leq d_A(I_1)$ and that $d_B(J_1)<d(A)$. This is our second case.

In the first case, we have that $d_A(I_1)\geq d(B)$ and $d_B(J_\ell)\geq d(A)$. Hence, by Corollary \ref{key bound} applied to $I_1, A$ and $B$, we have that 
\[\big|(B+(A\cap I_1))\cap(-\infty, \min A+\max B+d_R(B)]\big|\geq 2|B|\] 
and by the second part applied to $J_\ell,B$ and $A$, we have that
\[\big|(A+(B\cap J_\ell))\cap[\min A+\max B-d_L(A),\infty)\big|\geq 2|A|.\]
It follows that 
\[|A+B|\geq 2|A|+2|B|-d_L(A)-d_R(B)\geq 2|A|+2|B|-d(A)-d(B).\]

In the second case, let $J$ be the maximal interval that is left balanced with respect to $B$ and contains $\min B$. Let $j$ be such that $B\cap J=\bigcup_{i\leq j}(B\cap J_i)$. By Lemma \ref{partition} we have $d_B(J_i)\leq d_B(J_1)$ for every $i\leq j$, and by assumption $d_B(J_1)\leq d_A(I_1)$. Note also that $d(B\cap J)\leq d_B(J_1)$, since any discrepancy-maximizing interval $K$ for $B\cap J$ is contained in one of the $J_i$ with $i\leq j$, then $d_B(K)\leq d_B(J_i)$ since $J_i$ is balanced, and $d_B(J_i)\leq d_B(J_1)$.

We now apply Lemma \ref{main step} with $I$, $A$ and $B$ replaced by $-I_1$ and $-A$ and $B\cap J$, which we can do as $-I_1$ is balanced, and hence right balanced, with respect to $-A$. From it we conclude that
\begin{align*}
(B\cap J)+(I_1\cap A)&\supset\big[\min J+\min I_1, \min J+\min I_1+2|B\cap J|\big]\\
&=\big[\min A+\min B,\min A+\min B+2|B\cap J|\big].\\
\end{align*}

By Lemma \ref{partition}, $d_B(J_i)\leq d_B(J_1)$ for each $i\leq j$, which by hypothesis is less than $d(A)$, and hence than $d(B)$. It follows that $d(B)=d_B(J_t)$ for some $t>j$, and therefore that $d(B)=d(B\setminus J)$. 

By the inductive hypothesis, we know that
\[|A+(B\setminus J)|\geq 2|A|+2|B\setminus J|-2d(A).\]

The minimal element of $A+(B\setminus J)$ is $\min A+\min J_{j+1}$. We claim that this is greater than $\min A+\min B+2|B\cap J|$. Indeed, we know that $d_B(J')\geq 0$ for every $J'\subset J$ (where $J'$ is a left subinterval of $[\min B,\max B]$), and also that 
\[d_B(J')\geq 2|B\cap J|-|J'|\]
for every interval $J'\supset J$ that shares the left end-point with $J$. Thus, the interval $[\min B,\min B+2|B\cap J|]$ is left balanced, so if $\min J_{j+1}\leq\min B+2|B\cap J|$, then since $J_{j+1}$ is balanced, and hence left balanced, it follows that $[\min B,\max J_{j+1}]$ is left balanced, contradicting the maximality of $J$. 

It follows (using the hypothesis that $d(A)=d(B)$ again) that
\begin{align*}
|A+B|&\geq \big|\big[\min A+\min B,\min A+\min B+2|B\cap J|\big]\cup\big(A+(B\setminus J)\big)\big|\\
&=2|B\cap J|+|A+(B\setminus J)|\\
&\geq 2|B\cap J|+2|A|+2|B\setminus J|-2d(A)\\
&=2|A|+2|B|-d(A)-d(B),\\
\end{align*}
which completes the proof of the inductive step.
\end{proof}

We now obtain a version of Theorem \ref{offdiagonal main2} that does not require the condition $d(A)=d(B)$. We start with a simple technical lemma. We write $\{x\}$ for the fractional part $x-\lfloor x\rfloor$ of a real number $x$.

\begin{lemma} \label{reduce d}
    Given real numbers $\d>0$ and $0<\a\leq 1$, let $E_{\d,\a}$ be the set of all real numbers $x$ such that $0\leq\{x/\d\}\leq\a$. Then if $I$ is any interval, we have that
    \[\a|I|-\a\d\leq|I\cap E_{\d,\a}|\leq\a|I|+\a\d.\]
\end{lemma}

\begin{proof}
Let $s$ be the largest multiple of $\d$ that is less than or equal to $|I|$ and let $t$ be the smallest multiple of $\d$ that is greater than or equal to $|I|$. Let $I'$ be a subinterval of $I$ of length $s$ and let $I''$ be an interval of length $t$ that contains $I$. 

Then 
\[\a s=|I'\cap E_{\d,\a}|\leq|I\cap E_{\d,\a}|\leq|I''\cap E_{\d,\a}|=\a t.\]
Since $t-s\leq \d$ and $|I|\in[s,t]$, the result follows.
\end{proof}

\begin{corollary} \label{different discrepancies}
    Let $A$ and $B$ be unions of finitely many finite closed intervals and suppose that $d(A)\geq d(B)$. Then
    \[|A+B|\geq\Big(1+\frac{d(B)}{d(A)}\Big)|A|+2|B|-2d(B).\]
\end{corollary}

\begin{proof}
The result is trivial if $d(B)=0$, so we assume now that $d(B)>0$. We claim first that for every $\e>0$ there is a subset $\tilde A$ of $A$, which is a union of finitely many closed intervals, such that $d(\tilde A)=d(B)$ and $2|\tilde A|\geq\big(1+\frac{d(B)}{d(A)}\big)|A|-\e$. To prove this, we shall take $\tilde A$ to be the set $A\cap E_{\d,\a}$ for suitably chosen $\a$ and $\d$. 

Let the number of intervals that make up $A$ be $n$. Then for any interval $I$, $A\cap I$ is a union of at most $n$ closed intervals, so by Lemma \ref{reduce d},
\[\a|A\cap I|-n\a\d\leq|\tilde A\cap I|\leq\a|A\cap I|+n\a\d.\]
It follows that $|\tilde A|\geq\a|A|-n\a\d$. It also follows that
\begin{align*}
2|\tilde A\cap I|-|I|&\leq 2\a|A\cap I|-|I|+2n\a\d\\
&=(2-2\a)(|A\cap I|-|I|)+(2\a-1)(2|A\cap I|-|I|)+2n\a\d\\
&\leq (2\a-1)d(A)+2n\a\d.\\
\end{align*}
Therefore, if
\[(2\a-1)d(A)+2n\d\leq d(B),\]
or equivalently 
\[\a\leq\frac 12\Big(1+\frac{d(B)-2n\d}{d(A)}\Big),\]
it follows that $d(\tilde A)\leq d(B)$.

Also, if $\a=1$, then $\tilde A=A$, so in this case $d(\tilde A)=d(A)\geq d(B)$. Since $d(\tilde A)$ depends continuously on $\a$, it follows that for every $\d>0$, there exists $\a$ such that $d(\tilde A)=d(B)$ and $\frac 12\Big(1+\frac{d(B)-2n\d}{d(A)}\Big)\leq\a\leq 1$.

Since we can take $\d$ as small as we like, it follows that for every $\e>0$ we can choose $\tilde A$ with $2|\tilde A|\geq \Big(1+\frac{d(B)}{d(A)}\Big)|A|-\e$ and $d(\tilde A)=d(B)$, as claimed.

Once we have this set, the rest of the proof is easy. By Theorem \ref{offdiagonal main2}, 
\[|A+B|\geq|\tilde A+B|\geq 2|\tilde A|+2|B|-2d(B)\geq\Big(1+\frac{d(B)}{d(A)}\Big)|A|-\e+2|B|-2d(B).\]
Since $\e>0$ was arbitrary, the result follows.
\end{proof}

We remark that if $A$ and $B$ are intervals with $|A|\geq|B|$, then $|A+B|=|A|+|B|$, and since $d(A)=|A|$ and $d(B)=|B|$,
\[\Big(1+\frac{d(B)}{d(A)}\Big)|A|+2|B|-2d(B)=|A|+|B|+2|B|-2|B|=|A|+|B|\]
as well. Thus, Corollary \ref{different discrepancies} is sharp for this example.

\subsection{Sets with doubling constant less than 3}

We finish this section by showing how to recover from Theorem \ref{main2} the sharp symmetric stability result for the one-dimensional Brunn--Minkowski inequality discussed in the introduction.

The following amplification argument is reminiscent of the passage to a quotient modulo the diameter used in Ruzsa's work \cite{Ruzsa1991}. Its purpose is to reduce the effective discrepancy: while a single copy of $A$ may have discrepancy as large as $|A|$, a long string of translates has a much smaller discrepancy per copy, controlled asymptotically by the density of $A$ in its convex hull. This gain is what allows us to obtain the sharp stability bound from Theorem \ref{main2}.

\begin{corollary}
Let $A\subset\R$ be a compact set. Then
$$
    |A+A|-2|A|
    \geq
    \min\bigl\{
        |\operatorname{co}(A)\setminus A|,
        |A|
    \bigr\}.
$$
\end{corollary}
\begin{proof}
    If $|\operatorname{co}(A)|=0$, then $|A|=0$ and there is nothing to prove.
Up to scaling and translating, we may therefore assume that $A\subset[0,1]$
and that $\operatorname{co}(A)=[0,1]$. Since $A$ is compact, it follows that
$0,1\in A$.
For each positive integer $n$, let
$$
    A_n=A+\{0,1,\dots,n-1\}.
$$
Since the translates $A+j$ intersect only in sets of measure zero, we have that $|A_n|=n|A|$. We now want to bound the discrepancy of the set $A_n$.
We claim that
$$
    d(A_n)\leq n(2|A|-1)_+ +2,
$$
where $x_+ := \max\{x,0\}$ denotes the positive part of $x$.
Let $I$ be an interval. We may assume without loss of generality that $I\subseteq [0,n]$. Let $J$ be the largest subinterval of $I$ whose endpoints are integers. Then $|I\setminus J|\leq 2$. Moreover, on each interval of the form $[j,j+1]$, with $j\in{0,1,\dots,n-1}$, the set $A_n$ is a translate of $A$, which implies that $|A_n\cap J|=|A||J|$.
Therefore, 
$$
\begin{aligned}
2|A_n\cap I|-|I|
&\leq 2|A_n\cap J|+2|I\setminus J|-|J|-|I\setminus J|\\
&=(2|A|-1)|J|+|I\setminus J|\\
&\leq (2|A|-1)_+|J|+2\\
&\leq n(2|A|-1)_+ +2.
\end{aligned}
$$
Applying Theorem \ref{main2} to $A_n$, we obtain that
$$
    |A_n+A_n|
    \geq
    4|A_n|-2d(A_n)
    \geq
    4n|A|-2n(2|A|-1)_+-4.
$$
On the other hand,
$$
    A_n+A_n=\bigcup_{j=0}^{2n-2}(A+A+j).
$$
Since $0,1\in A$, for each $1\leq j\leq 2n-2$ we have
$$
    A+j\subset (A+A+j-1)\cap(A+A+j).
$$
It follows that when we add the sets
$A+A+j$ one at a time, each new set overlaps the previous union in measure at
least $|A|$. Therefore,
$$
    |A_n+A_n|
    \leq
    (2n-1)|A+A|-(2n-2)|A|.
$$
Combining the two bounds for $|A_n+A_n|$ and letting $n\to\infty$, we obtain
$$
4|A|-2(2|A|-1)_+
\leq
2|A+A|-2|A|.
$$
Finally, using $|A|-(2|A|-1)_+=\min\{1-|A|,|A|\}$, the last inequality can be rewritten as
$$
    |A+A|-2|A|
    \geq
    \min\bigl\{
        |\operatorname{co}(A)\setminus A|,
        |A|
    \bigr\},
$$
as wanted.
\end{proof}

Note that the above result is trivial if $|A+A|\geq 3|A|$, but that if $|A+A|=C|A|$ for some $C<3$, then it proves that $A$ has diameter at most $(C-1)|A|$.

\section{Proof of Lemma \ref{key lemma} and hence Theorem \ref{main1}}

Recall that Lemma \ref{key lemma} states that if $A$ is a sum-free set which is a finite union of closed intervals and has minimum equal to $1$, then, defining $\F{x} = |A\cap [0,x]|$ (which by hypothesis is equal to $|A\cap [1,x]|$), we have for every $x\geq 1$ the inequality
\begin{equation*}
\F{x-1}+\F{x}+\F{x+1}\leq x .
\end{equation*}
Recall also that this implies our first main theorem, that a product-free open subset of $(0,1)$ has measure at most 1/3. 

\subsection{Overview of the proof}
Let us begin by recalling that the inequality
$$
\F{u-1}+\F{u}+\F{u+1}\leq u
$$
is sharp, with equality when
$$
A=[1,2)\cup[4,5)\cup[7,8)\cup\cdots.
$$
This example should be kept in mind throughout the proof. It has period $3$: in each interval of length $3$, the set occupies one interval of length $1$, and the remaining two units form the gap before the next occupied block. The natural points at which equality occurs are the points $x=3m$. At such a point, $x-1$ and $x+1$ lie in the neighbouring occupied blocks, with residues $2$ and $1$ modulo $3$ respectively. Thus, if the inequality were false, one should expect the set near the relevant point $x$ to resemble this period-3 configuration.

We therefore argue by contradiction and choose a point $x$ at which the inequality fails and a preliminary lemma shows that this may be done so that $x-1$ and $x+1$ both belong to $A$. Thus the local picture is already the one suggested by the sharp example: the point $x$ lies in a gap, while the set appears on either side of it. One should think of $x$ as morally playing the role of a multiple of $3$. The rest of the argument shows that more and more of the period-3 rigidity must also be present.

We then introduce a discrepancy-maximizing interval $I=[a,b]\subseteq[0,x+1]$,
and write
$$
s=b-a,\qquad t=d(A\cap[0,x+1]),\qquad y=x+1-a.
$$
The interval $I$ is chosen to maximize the discrepancy of $A\cap[0,x+1]$, and therefore identifies the part of $A$ that carries the main structure at this scale. The parameter $y$ measures the position of this block relative to the point $x+1$. In the model example, the relevant block is the one immediately to the left of $x+1$, whose left endpoint lies exactly three units earlier. A central part of the proof is to show that the same remains true in any hypothetical counterexample.

We begin by showing that the discrepancy-maximizing interval $I=[a,b]$ cannot come too close to $x$. The first step is to prove that $A$ is empty on the interval $[y-1,y-s]$. In the sharp example this interval forms part of the gap between two consecutive occupied blocks, so this is the first indication that the same period-3 geometry is being forced in a hypothetical counterexample. From this we deduce first that $b<x$, equivalently that $y>1+s$, and then that $a<x-1$, equivalently that $y>2$. Thus the whole interval $I$ is forced to lie strictly to the left of the local configuration around $x$.

The next step is to show that this first gap has the correct rigidity. If $u$ denotes the last point of $A$ before $y-1$ and $v$ the first point of $A$ after $y-1$, then we prove that $v-u\geq 1$. In other words, $y-1$ lies inside a gap of length at least $1$. Again this is exactly what happens in the sharp example: once one sees one missing block, it should already have essentially the same scale as the occupied one. This estimate is then used in a sequence of difference-set arguments that push the maximizing interval further and further to the left. First one obtains $y>2+s$, and then, after a more delicate argument, one reaches the much sharper conclusion that $3<y<3+s$. At that point the position of $I$ is almost forced: its left endpoint must lie very close to $x-2$, just as in the extremal period-3 configuration.

Once $y$ has been confined to this narrow range, we turn back to the region around $x$ itself. The argument then shows that there is also a gap of length at least $1$ around $x$, and from this one derives the important lower bound $t>1/2$. This says that the discrepancy-maximizing block is not only well positioned, but also genuinely dense. In the model example this is again the correct picture, since the relevant occupied block carries more than half of the mass of its ambient unit interval.

The next step is to show that the structure already forced near $x$ also has consequences near the beginning of the set. Using the information already obtained, we prove that the point $2$ must also lie in a large gap of $A$. Thus the period-3 geometry is no longer confined to the neighbourhood of $x$: the same constraints now force a large gap at the beginning of the configuration, around the point $2$. A final bootstrapping argument then shows that the gap around $y-1$ and the gap around $2$ must in fact line up too closely to be compatible with sum-freeness.

At that point there is no room left. The various gap estimates, density estimates, and sum-set and difference-set bounds become incompatible, and the assumed counterexample cannot exist. This contradiction completes the proof.

To make this strategy work, we repeatedly use the tools developed in the preceding sections. The argument is delicate because the inequalities used along the way are essentially sharp in the extremal configuration, and therefore leave very little room for loss. The proof is consequently arranged as a rigidity argument: starting from a hypothetical failure of the desired inequality, we force the set to resemble the period-3 extremal example more and more closely, until the accumulated gap, density, and difference-set constraints become incompatible.

\subsection{Preparatory lemmas}

Before we embark on the proof in earnest, we prove some general facts that will be useful in the proof, some of them repeatedly.

\begin{lem} \label{s<1} Let $A$ be a compact set of real numbers such that $A\cap(A-1)=\emptyset$ and let $I$ be a discrepancy-maximizing interval for $A$. Then $|I|<1$.
\end{lem}

\begin{proof}
Let $I=[a,b]$. If $b\geq a+2$, then let $A_1=A\cap[a,a+1]$ and $A_2=A\cap[a+1,a+2]$. Then $A_1$ and $A_2-1$ are disjoint closed subsets of $[a,a+1]$, so since $[a,a+1]$ is connected, it follows that $|A_1|+|A_2|<1$ and therefore that $d_A([a,a+1])+d_A([a+1,a+2])<0$. It follows that $I$ is not balanced, which contradicts the fact that a discrepancy-maximizing interval is balanced.

If $a+1< b<a+2$, then let $A_1=A\cap[a,b-1]$, $A_2=A\cap[b-1,a+1]$ and $A_3=A\cap[a+1,b]$. From the same argument, as $A_1$ and $A_3-1$ are disjoint, we obtain that $|A_1|+|A_3|<b-a-1$ and therefore that $d_A([a,b-1])+d_A([a+1,b])<0$, which again contradicts the fact that a discrepancy-maximizing interval is balanced. 

Finally, if $b=a+1$, then $a$ and $b$ cannot both belong to $A$, since
$A\cap(A-1)=\emptyset$. But a balanced interval of positive length must contain both of its endpoints. Thus $I$ is not balanced, once again contradicting discrepancy-maximality.
\end{proof}

For the next lemma we introduce a piece of notation that will be used throughout this section. Let $A$ be a measurable subset of $\R$ and let $z\in\R$. Then we set
\[m_A^-(z)=\sup_{x\leq z}d_A([x,z])\]
and
\[m_A^+(z)=\sup_{x\geq z}d_A([z,x]).\]
In words, $m_A^-(z)$ is the largest discrepancy, with respect to $A$, of any interval with $z$ as its right end-point, and $m_A^+(z)$ is the largest discrepancy of an interval with $z$ as its left end-point. Note that if $z=\max A$, then $m_A^-(z)$ is the right-discrepancy $d_R(A)$ of $A$, and if $z=\min A$ then $m_A^+(z)$ is the left-discrepancy $d_L(A)$.

The following lemma will be useful for showing that $m_A^-(z)=0$ in a number of situations.
\begin{lem}\label{mzero}
Let $A$ be a sum-free union of finitely many finite closed intervals, let $0\leq z\leq x+1$, and let $I=[a,b]$ be a discrepancy-maximizing interval for $A\cap[0,x+1]$. Then $m_A^-(z)=0$ whenever $a+z\in A$ or $b-z\in A$.
\end{lem}
\begin{proof}
         Since $I$ is discrepancy maximizing, it is both left balanced and
    right balanced with respect to $A$. Suppose, for a contradiction, that $m_A^-(z)>0$. Choose $r\leq z$ such that
    $$
    2|A\cap [r,z]|-(z-r)=m_A^-(z)>0.
    $$
    By the choice of $r$, the interval $J=[r,z]$ is left balanced. Moreover, since $I$ is discrepancy maximizing, we have $\disc_A(I)\geq d(A\cap J)$. Applying Lemma \ref{main step} to $I$ and $J$ with $B=A\cap J$ gives
$$
    (A\cap J)-(A\cap I)
    \supseteq [r-b,r-b+2|A\cap J|].
$$
Reflecting, we obtain
$$
    (A\cap I)-(A\cap J)
    \supseteq [b-r-2|A\cap J|,b-r].
$$
Similarly, applying Corollary \ref{main step sums} gives
$$
    (A\cap I)+ (A\cap J)
    \supseteq [a+r,a+r+2|A\cap J|].
$$
    Since $2|A\cap J|>z-r$, the first interval contains $b-z$, while the second interval contains $a+z$, contradicting the sum-free hypothesis. Therefore,   $m_A^-(z)=0$.
\end{proof}
The following lemma allows us to obtain lower bounds for the portions of $A+A$ and $A-A$ that lie in certain intervals determined by the discrepancy-maximizing interval $I$ and the quantity $m_A^-(\beta)$.
\begin{lem}\label{minterv}
Let $A$ be a finite union of finite closed intervals, let $I=[a,b]$ be discrepancy maximizing with respect to $A\cap[0,x+1]$,
and let $[\alpha,\beta]\subseteq[0,x+1]$. Then
$$
    |(A-A)\cap [b-\beta-m_A^-(\beta),\, b-\alpha]|
    \geq 2|A\cap[\alpha,\beta]|=2(\F{\beta}-\F{\alpha}).
$$
The corresponding sumset estimate is
$$
    |(A+A)\cap [a+\alpha,\, a+\beta+m_A^-(\beta)]|
    \geq 2|A\cap[\alpha,\beta]|=2(\F{\beta}-\F{\alpha}).
$$
In particular, if $m_A^-(\beta)=0$, then
$$
    |(A-A)\cap [b-\beta,\, b-\alpha]|
    \geq 2|A\cap[\alpha,\beta]|
$$
and
$$
    |(A+A)\cap [a+\alpha,\, a+\beta]|
    \geq 2|A\cap[\alpha,\beta]|.
$$
\end{lem}

\begin{proof}
The two statements are direct consequences of Corollary \ref{key bound}. We begin with the sum-set part. To prove this we apply the first part of Corollary \ref{key bound} to $I$ and $A$, taking $B$ to be $A\cap[\a,\be]$. The conditions of Corollary \ref{key bound} are easily seen to hold, and the conclusion is that
\[\big|(B+(A\cap I))\cap[a+\min B,a+\max B+d_R(B)]\big|\geq 2|B|.\]
But $B+(A\cap I)\subset A+A$, $\min B\geq\a$, $d_R(B)=m_A^-(\max B)\leq m_A^-(\beta)+\beta-\max B$, and $2|B|=2|A\cap[\a,\be]|=2(\F{\be}-\F{\a})$, so we obtain the desired estimate.

For the difference-set part we apply the second part of Corollary \ref{key bound}, this time with $B=-(A\cap[\a,\be])$. Again the conditions are easily seen to hold, and this time the conclusion is that
\[\big|(B+(A\cap I))\cap[b+\min B-d_L(B), b+\max B]\big|\geq 2|B|.\]
Now we have that $B+(A\cap I)\subset A-A$ and $\max B\leq-\a$. Also, $d_L(B)=d_R(-B)$, which we saw above was at most $m_A^-(\be)+\beta-\max(-B)=m_A^-(\be)+\beta+\min B$, and of course $|B|$ is still equal to $|A\cap[\a,\be]|$. Therefore, we obtain the desired estimate in this case too.
\end{proof}

We shall use the following simple consequence of sum-freeness several times:
if $J$ is an interval of length at most $2$, then $|A\cap J|\leq 1$.
Indeed, it is enough to consider the case $J=[r,r+2]$, and then the sets
$A\cap[r,r+1]$ and $(A\cap[r+1,r+2])-1$ are disjoint subsets of $[r,r+1]$.
\subsection{Setup and the first gap}

To prove Lemma \ref{key lemma}, we shall argue by contradiction. Let us fix, once and for all in this section, a finite union $A$ of finite closed intervals with $\min A =1$, and let us fix $f$ to be the corresponding distribution function of the measure $\mathbbm 1_A(x)dx$.
 
\begin{lem}\label{positive-triple-reduction} If there exists $z\geq 1$ such that 
$$ 
\F{z-1}+\F{z}+\F{z+1}>z, 
$$
then there exist $x\geq 1$ and $\epsilon>0$ such that
$$ 
\F{y-1}+\F{y}+\F{y+1}>y \qquad\text{for every } y\in[x-\epsilon,x+\epsilon], 
$$ 
and 
$$ 
[x-1-\epsilon,x-1+\epsilon]\cup[x+1-\epsilon,x+1+\epsilon]\subseteq A. 
$$ 
\end{lem}
\begin{proof}
     Let
    $$
        H(x)=\F{x-1}+\F{x}+\F{x+1}-x.
    $$
    Since $f$ is $1$-Lipschitz, $H$ is Lipschitz, and therefore absolutely continuous on every compact interval. Moreover, for almost every $x$,
    $$
        H'(x)=1_A(x-1)+1_A(x)+1_A(x+1)-1.
    $$
    Since $A$ has its minimum in $1$, we have $H(1)\leq 0$.
    Choose $z\geq 1$ with $H(z)>0$, and let
    $$
        \alpha=\sup\{r\in[1,z]:H(r)\leq 0\}.
    $$
    Then $H(\alpha)=0$ and $H>0$ on $(\alpha,z]$.
    Suppose, for a contradiction, that the set of $x\in(\alpha,z)$ for which at least two among $x-1,x,x+1$ belong to $A$ has measure zero. Then $H'\leq 0$ almost everywhere on $(\alpha,z)$. Hence, for every $\alpha<r<z$,
    $$
        H(z)-H(r)=\int_r^z H'(u)\,du\leq 0.
    $$
    Letting $r\to \alpha$, we get $H(z)\leq H(\alpha)=0$, a contradiction.
    Therefore,   the set of $x\in(\alpha,z)$ for which at least two among $x-1,x,x+1$ belong to $A$ has positive measure. Since $A$ is sum free and $1\in A$, the only possible pair is $x-1,x+1$.
    Since $A$ is a finite union of intervals, the set
$$
    \{x\in(\alpha,z):x-1\in A,\ x+1\in A\}
    =
    (\alpha,z)\cap(A+1)\cap(A-1)
$$
is a finite union of intervals. Since it has positive measure, it contains an interval of positive length and this concludes the proof.
\end{proof}

For the rest of the section, let us fix the following.
\begin{itemize}
\item $x$ and $\e$ are positive real numbers with the property guaranteed by Lemma \ref{positive-triple-reduction} (as well as the property that $x-1, x+1\in A$, which is given by the proof). 
\item $t=d(A\cap[0,x+1])$.
\item $I=[a,b]$ is a subinterval of $[0,x+1]$ with $\disc_A(I)=t$.
\item $s=b-a$ is the length of $I$.
\item $y=x+1-a$ is the distance from $\min I$ to $\max(A\cap[0,x+1])$.
\end{itemize}
Since $A$ is a union of finitely many closed intervals, we may choose a discrepancy-maximizing interval $I=[a,b]$ such that $a,b\in A$. Also, since $I$ is a discrepancy-maximizing interval and $A\cap(A-1)=\emptyset$, Lemma \ref{s<1} implies that $s<1$. Thus, $0<t\leq s<1$.

The next lemma shows that the failure of the desired inequality forces some mass to lie in the interval $[x,x+1]$. In the extremal period-3 example, where one should think of $t$ as being essentially $1$, the resulting surplus $1-t$ is morally zero. Even so, it will be enough to force the first gap.

\begin{lem} \label{fx+1-fx}
$\F{x+1}-\F{x}>1-t$.
\end{lem}

\begin{proof}
Applying Theorem \ref{offdiagonal main2} to $A\cap[1,x]$ and $-(A\cap[1,x])$ and using the central symmetry of $(A\cap[1,x])-(A\cap[1,x])$, we have that 
\begin{equation*}
    |(A-A)\cap [0,x-1]|\geq 2|A\cap[1,x]|-t=2\F{x}-t,
\end{equation*}
and since $A$ is sum free we have that $x-1\geq |(A-A)\cap [0,x-1]|+\F{x-1}$. Therefore,
\begin{equation*}
    \F{x-1}+2\F{x}\leq x-1+t.
\end{equation*}
Since $\F{x-1}+\F{x}+\F{x+1}>x$, this implies that $\F{x + 1} - \F{x} > 1 - t$.
\end{proof}
We now use the excess mass in $[x,x+1]$ to force the first gap. Any point of $A$ in $[y-1,y-s]$ would translate $A\cap I$ into $[x,x+1]$, leaving too little room for the mass guaranteed by the previous lemma.
\begin{lem}\label{gapy}
    $A\cap[y-1,y-s]=\varnothing$.
\end{lem}
\begin{proof}
Suppose that $r\in A\cap[y-1,y-s]$. Since $r+a\geq x$ and $r+b\leq x+1$, we have
$$
r+(A\cap I)\subseteq [x,x+1].
$$
Since $r\in A$ and $A$ is sum free, the set $r+(A\cap I)$ is disjoint from $A$, and therefore
$$
|[x,x+1]\setminus A|\geq |A\cap I|=\frac{s+t}{2}\geq t.
$$
But $\F{x+1}-\F{x}>1-t$, by Lemma \ref{fx+1-fx}, so $|[x,x+1]\setminus A|<t$, a contradiction.
\end{proof}

\subsection{Pushing the maximizing interval away from $x$}

In this subsection we show that the discrepancy-maximizing interval $I=[a,b]$ must lie a definite distance to the left of $x$. In particular, neither $a$ nor $b$ can be too close to $x$. We also show that $y-1$ lies in a gap of length at least $1$.
\begin{lem}
    We have $y > 1 + s$. Equivalently, $b < x$. 
\end{lem}
\begin{proof}
    Suppose, for a contradiction, that $b\geq x$. By Lemma \ref{mzero},
    we have $m_A^-(b-1)=0$, since $b-(b-1)=1\in A$. Therefore, by
    Lemma \ref{minterv} with $[\alpha,\beta]=[1,b-1]$,
    $$
        |(A-A)\cap[1,b-1]|\geq 2\F{b-1}.
    $$
    Since $A$ is sum free, $(A-A)\cap[1,b-1]$ is disjoint from
    $A\cap[1,b-1]$. Hence,  
    $$
        3\F{b-1}\leq b-2.
    $$
    Since $b\geq x$, the interval $[b-1,x+1]$ has length at most $2$.
    As $1\in A$ and $A$ is sum free, we have $A\cap(A+1)=\varnothing$,
    and therefore
    $$
        |A\cap[b-1,x+1]|\leq 1.
    $$
    Thus
    $$
        2\F{b-1}+\F{x+1}\leq b-1.
    $$
Also, from the definition of $f$ and the inequality $x\geq b-1$, we have
$$
    \F{x}\leq \F{b-1}+x-b+1.
$$
    Therefore
    $$
        \F{x-1}+\F{x}+\F{x+1}
        \leq \F{b-1}+\F{x}+\F{x+1}
        \leq 2\F{b-1}+\F{x+1}+x-b+1
        \leq x,
    $$
    contradicting our choice of $x$.
\end{proof}
The inequality $y>1+s$ is equivalent to $b<x$, so it already shows that the discrepancy-maximizing interval does not reach the local configuration around $x$. The next step is to strengthen this to $y>2$, or equivalently $a<x-1$, which places the whole interval $I$ strictly to the left of $[x-1,x+1]$.
\begin{lem}
    $y > 2$. 
\end{lem}
\begin{proof}
    By the previous lemma, $y > 1 + s$. If $y\leq 2$, then $1\in [y - 1, y - s]$, contradicting $1\in A$ and $A\cap [y - 1, y - s]=\varnothing$.
\end{proof}

We now sharpen Lemma \ref{gapy} by showing that $y-1$ lies in a gap of length at least $1$ in $A$. In view of Lemma \ref{gapy}, it will follow that $y-s$ does as well. Let
$$
u=\max(A\cap(-\infty,y-1]),\qquad v=\min(A\cap [y-1,\infty)).
$$

\begin{lem} \label{y - 1 gap lem}
    $v - u \geq 1$. 
\end{lem}
\begin{proof}
    Suppose, for a contradiction, that $v<u+1$. Since
$A\cap[y-1,y-s]=\varnothing$, we have $u<y-1$ and $v>y-s$.
Hence, from $v<u+1$, we get
$ u>y-s-1$.
Since $u,v\in A$ and $A$ is sum free, the sets
$$
u+(A\cap I)
\qquad\text{and}\qquad
v+(A\cap I)
$$
are disjoint from $A$. Therefore,
$$
\big|[x,x+1]\setminus ((u+(A\cap I))\cup (v+(A\cap I)))\big|
\geq \F{x+1}-\F{x}>1-t.
$$
Moreover, since $|A\cap I|=(s+t)/2$, we have $|I\backslash A|=(s-t)/2$. Thus passing from $u+(A\cap I)$ and $v+(A\cap I)$ to the full intervals $u+I$ and $v+I$ can decrease the complement in $[x,x+1]$ by at most $s-t$.
   Hence,  
    \begin{equation*}
        |[x, x + 1]\backslash([a + u, b + u]\cup [a + v, b + v])| > 1 - s. 
    \end{equation*}
 Since $a+u<x<b+u$ and $a+v<x+1<b+v$, the uncovered part of $[x,x+1]$ is contained in the gap between $b+u$ and $a+v$. Hence,  
$$
    (a+v)-(b+u)>1-s.
$$
But $(a+v)-(b+u)=v-u-s<1-s$, a contradiction.
\end{proof}
As $y > 2$, we have $u\geq 1$ and hence $v\geq 2$. The gap estimate $v-u\geq 1$ allows us to strengthen the previous lower bounds for the portion of $A-A$ contained in $[b-d,b-1]$. The next lemma is set up so that, when applied later with $c=u$, the term involving $A\cap[c,c+s]$ will disappear. 
\begin{lem} \label{another boost lem}
If $m_A^-(d) = 0$, $d\leq x+1$, $c\in A$ and $1\leq c\leq d - s$, then
$$
|(A - A)\cap [b - d, b - 1]|\geq 2\F{d} + |A\cap [a, b - m_A^-(c)]| - 2|A\cap [c, c + s]|.
$$
\end{lem}

\begin{proof}
Since $I$ is discrepancy maximizing, we have $m_A^-(c)\leq t\leq s$. By Lemma \ref{minterv} applied to $[1,c]$ we have
$$
|(A - A)\cap [b - c - m_A^-(c), b - 1]|\geq 2\F{c}.
$$
Also, since $m_A^-(d)=0$, applying Lemma \ref{minterv} to $[c+s,d]$ gives
$$
|(A - A)\cap [b - d, b - c - s]|\geq 2(\F{d} - \F{c+s}).
$$
Finally, since $c\in A$ and $a=b-s$,
$$
(A-c)\cap [b-c-s,b-c-m_A^-(c)]
=
(A\cap [a,b-m_A^-(c)])-c
\subseteq A-A.
$$
The intervals
$$
[b-d,b-c-s],\qquad [b-c-s,b-c-m_A^-(c)],\qquad [b-c-m_A^-(c),b-1]
$$
are disjoint up to endpoints and are contained in $[b-d,b-1]$. 
Hence, summing the three estimates above,
$$
|(A - A)\cap [b - d, b - 1]|
\geq
2(\F{d}-\F{c+s})+|A\cap [a,b-m_A^-(c)]|+2\F{c}.
$$
Since
$$
2(\F{d}-\F{c+s})+2\F{c}
=
2\F{d}-2|A\cap[c,c+s]|,
$$
we get
$$
|(A - A)\cap [b - d, b - 1]|
\geq
2\F{d}+|A\cap [a,b-m_A^-(c)]|-2|A\cap[c,c+s]|.
$$
\end{proof}

In particular, we have the following corollary. 

\begin{cor} \label{another boost cor}
    If $m_A^-(d) = 0$ and $x+1\geq d\geq u + s$, then
    \begin{equation*}
        |(A - A)\cap [b - d, b - 1]|\geq 2\F{d} + |A\cap [a, b - m_A^-(u)]|. 
    \end{equation*}
\end{cor}
\begin{proof}
Since $y>2$, we have $u\geq 1$. Moreover, $u\in A$, and the assumption $d\geq u+s$ gives $u\leq d-s$. Hence,   Lemma \ref{another boost lem} applies with $c=u$, and gives
$$
|(A - A)\cap [b - d, b - 1]|
\geq
2\F{d} + |A\cap [a, b - m_A^-(u)]| - 2|A\cap [u,u+s]|.
$$
Since $v-u\geq 1$ and $s\leq 1$, we have $u+s\leq v$. By the definition of $u$ and $v$, the interval $(u,v)$ is disjoint from $A$, and therefore $|A\cap [u,u+s]|=0$. This gives
$$
|(A - A)\cap [b - d, b - 1]|
\geq
2\F{d} + |A\cap [a, b - m_A^-(u)]|.
$$
\end{proof}
\subsection{Pushing the maximizing interval to the left of $x-1$}

In this subsection we prove the stronger estimate $y>2+s$. Along the way, we establish several auxiliary lemmas that hold under the additional hypothesis $y<3+s$, and these will also be used later in the proof.

The regime $y<3+s$ is equivalent to having $b>x-2$, so the discrepancy-maximizing interval $I=[a,b]$ already reaches into the interval immediately to the left of $x$. This will allow us to extract additional information about the position of $I$.
\begin{lem}
\label{small case} If $y<3+s$, then $a\geq v$. 
\end{lem}
\begin{proof}
Suppose, for a contradiction, that $a<v$. Since $a\in A$ and $v=\min(A\cap[y-1,\infty))$, we have $a<y-1$, and hence $a\leq u$ by definition of $u$. We claim that $b\leq u$. Suppose, for a contradiction, that $b>u$. Since $b\in A$ and $(u,v)\cap A=\varnothing$, we have $b\geq v$. Hence,   $[u,v]\subseteq I$. But $v-u\geq1$ and $s\leq1$, so this forces $I=[u,v]$, which is impossible since $(u,v)\cap A=\varnothing$ whereas $I$ has positive discrepancy. Therefore,   $b\leq u$. Note that by Corollary \ref{add balanced} we have the inclusion $[2a, 2b]\subseteq A + A$, moreover $[2a,2b]$ is disjoint from $A$ since the set is sum-free.
Therefore, either $2a > x + 1$ or $2b < x + 1$. Since $y - 1\geq u \geq b$, we have $2a \leq x + 1$. Therefore, $2b < x + 1$. Similarly, 
\begin{equation*}
    2b = 2a + 2s\geq a + 2s + 1 = x - y + 2s + 2\geq x - 1,
\end{equation*}
and since $x-1\in A$ while $[2a,2b]$ is disjoint from $A$, this implies $2a > x - 1$. Hence,   $[2a, 2b]\subseteq [x - 1, x + 1]$.
Note that
$$
    x-1=a+y-2\leq u-s+y-2<u+1\leq v.
$$
Choose $0<\epsilon'\leq\epsilon$ such that $x-1+\epsilon'<v$. Since
$x-1+\epsilon'\in A$ and $v=\min(A\cap[y-1,\infty))$, we must have
$$
    x-1+\epsilon'<y-1.
$$
Hence, $x+\epsilon'<y$, and therefore $a<1$, a contradiction.
\end{proof}

\begin{cor}
    If $y < 3 + s$, then
    $$
        3\F{b - 1}\leq b - 2 - \frac{s + t}{2} + m_A^-(u).
    $$
\end{cor}
\begin{proof}
     By Lemma \ref{small case}, we have $a\geq v$. Since $v-u\geq 1$, this gives
    $$
        b-1=a+s-1\geq v+s-1\geq u+s.
    $$
    Hence, Corollary \ref{another boost cor}, applied with $d=b-1$, gives
    $$
        |(A-A)\cap[1,b-1]|
        \geq
        2\F{b-1}+|A\cap[a,b-m_A^-(u)]|,
    $$
    where we used $m_A^-(b-1)=0$. Since $A$ is sum free, we can add the mass of $A\cap[1,b-1]$ inside $[1,b-1]$, and therefore
    $$
        3\F{b-1}+|A\cap[a,b-m_A^-(u)]|\leq b-2.
    $$
    Finally, since $m_A^-(u)\leq t\leq s$, we have $b-m_A^-(u)\in[a,b]$. Hence,  
$$
    |A\cap[a,b]|
    \leq
    |A\cap[a,b-m_A^-(u)]|+m_A^-(u),
$$
because passing from $[a,b]$ to $[a,b-m_A^-(u)]$ removes only an interval of length $m_A^-(u)$. Therefore
$$
    |A\cap[a,b-m_A^-(u)]|
    \geq
    |A\cap[a,b]|-m_A^-(u)
    =
    \frac{s+t}{2}-m_A^-(u).
$$
Substituting this into
$$
    3\F{b-1}+|A\cap[a,b-m_A^-(u)]|\leq b-2
$$
gives
$$
    3\F{b-1}
    \leq
    b-2-\frac{s+t}{2}+m_A^-(u).
$$
\end{proof}

The next lemma is proved by contradiction. Under the hypothesis $y\leq 2+s$, the previous corollary gives an improved bound for $\F{b-1}$, while the inequality $b+1\geq x$ allows one to control the contribution of $A$ in $[b+1,x+1]$ by means of sum-freeness and the Brunn--Minkowski inequality. Together, these estimates force
$$
\F{x-1}+\F{x}+\F{x+1}\leq x,
$$
contradicting the choice of $x$.
\begin{lem}
    We have $y > 2 + s$. Equivalently, $b < x - 1$. 
\end{lem}
\begin{proof}
    For the sake of contradiction, let $y\leq 2 + s$. By the preceding corollary, 
    \begin{equation*}
        3\F{b - 1}\leq b - 2 - \frac{s + t}{2} + m_A^-(u). 
    \end{equation*}
Also, by the maximality of the discrepancy of $I$, applied to the interval $[b-1,x-1]$, we have 
$$ 
\F{x-1}-\F{b-1}\leq \frac{x-b+t}{2}. 
$$ 
Adding the last two inequalities gives 
$$
\F{x-1}+2\F{b-1} \leq x-\frac{y+3}{2}+m_A^-(u), 
$$ 
where we used $b=x+1-y+s$.

We next estimate the mass of $A$ in $[b-1,b+1]$. First, 
$$ 
\F{b+1}-\F{b-1} \leq s+(\F{a}-\F{b-1})+(\F{a+1}-\F{b}). 
$$ 
Indeed, the two intervals $[a,b]$ and $[a+1,b+1]$ have length $s$, and since $A$ is sum free and $1\in A$, the contribution of $A$ in these two intervals is at most $s$. The remaining parts of $[b-1,b+1]$ are $[b-1,a]$ and $[b,a+1]$, which give the two other terms. By the maximality of the discrepancy of $I$ applied to $[b-1,b]$, we have
$$
2(\F{b}-\F{b-1})-1\leq t.
$$
Subtracting $2(\F{b}-\F{a})-s=t$, we get
$$
\F{a}-\F{b-1}\leq \frac{1-s}{2}.
$$
Hence, 
$$ 
\F{b+1}-\F{b-1} \leq \frac{1+s}{2}+\F{a+1}-\F{b}. 
$$ 
Combining this with the previous estimate gives 
$$ 
\F{x-1}+2\F{b+1} \leq x-\frac{y+1}{2}+m_A^-(u)+s+2(\F{a+1}-\F{b}). $$
We now use the sum-free condition together with Brunn--Minkowski. Since
$$
    (A\cap[0,y-1])+(A\cap[b,a+1])
    \subseteq [b+1,x+1],
$$
and this sumset is disjoint from $A$, we have
$$
    |A\cap[b+1,x+1]|
    +
    |(A\cap[0,y-1])+(A\cap[b,a+1])|
    \leq x-b.
$$
By the Brunn--Minkowski inequality, 
$$
    |(A\cap[0,y-1])+(A\cap[b,a+1])|
    \geq
    |A\cap[0,y-1]|+|A\cap[b,a+1]|.
$$
Therefore
$$
    \F{x+1}-\F{b+1}+\F{y-1}+(\F{a+1}-\F{b})\leq x-b.
$$
Since $y\leq 2+s$, we have
$$
    b+1=x+2-y+s\geq x,
$$
and hence $\F{x}\leq \F{b+1}$. Combining the last estimate with the bound
$$
    \F{x-1}+2\F{b+1}
    \leq
    x-\frac{y+1}{2}+m_A^-(u)+s+2(\F{a+1}-\F{b}),
$$
we get
$$
\begin{aligned}
    \F{x-1}+\F{x}+\F{x+1}
    &\leq \F{x-1}+\F{b+1}+\F{x+1} \\
    &\leq \F{x-1}+2\F{b+1}+x-b-\F{y-1}-(\F{a+1}-\F{b}) \\
    &\leq 2x-b-\frac{y+1}{2}+s+m_A^-(u)-\F{y-1}+\F{a+1}-\F{b}.
\end{aligned}
$$
Since $b=x+1-y+s$, this gives
$$
    \F{x-1}+\F{x}+\F{x+1}
    \leq
    x+\frac{y-3}{2}+m_A^-(u)-\F{y-1}+\F{a+1}-\F{b}.
$$
Now $m_A^-(u)\leq \F{u}\leq \F{y-1}$, so
$$
    \F{x-1}+\F{x}+\F{x+1}
    \leq
    x+\frac{y-3}{2}+\F{a+1}-\F{b}.
$$
Finally, by the maximality of the discrepancy of $I$, applied to $[a,a+1]$,
$$
    \F{a+1}-\F{b}\leq \frac{1-s}{2}.
$$
Therefore
$$
    \F{x-1}+\F{x}+\F{x+1}
    \leq
    x+\frac{y-3}{2}+\frac{1-s}{2}
    =
    x+\frac{y-s-2}{2}
    \leq x,
$$
because $y\leq 2+s$. This contradicts the choice of $x$.
\end{proof}

\subsection{Excluding the range $y\leq 3$}

We now prove that in fact $y>3$. We argue by contradiction, assuming that $y\leq 3$, and split into two cases depending on whether $A\cap [b+1,x]$ is empty or not. In both cases, the information already obtained about the gap around $y-1$ and the position of the discrepancy-maximizing interval $I=[a,b]$ yields a contradiction.
\begin{lem} \label{y > 3 lem}
   We have $y > 3$. 
\end{lem}
\begin{proof}
    Assume, for a contradiction, that $y\leq 3$. We already know that $y>2+s$, and hence $b<x-1$.

We first consider the case where $A$ intersects $[b+1,x]$. Choose $z\in A\cap[b+1,x]$. Then
$$
z-(A\cap I)\subseteq [z-b,z-a]\subseteq [1,y-1].
$$
Since $A$ is sum free, this set is disjoint from $A$.
Therefore
$$
\F{u}=\F{y-1}\leq y-2-\frac{s+t}{2}.
$$
We shall also use that
$$
\F{y}-\F{u}\leq \frac{s-t}{2}.
$$
Indeed, since $A\cap[y-1,y-s]=\varnothing$, we have
$\F{y}-\F{u}=|A\cap[y-s,y]|$. If $r\in A\cap[y-s,y]$, then $x+1-r\in I$; and since $x+1\in A$, the point $x+1-r$ cannot belong to $A$. Hence,   reflection through $x+1$ sends $A\cap[y-s,y]$ into $I\setminus A$, and so
$$
\F{y}-\F{u}\leq |I\setminus A|=\frac{s-t}{2}.
$$
By Theorem \ref{offdiagonal main2}, applied to $A\cap[y,x+1]$, and using
$d(A\cap[y,x+1])\leq t$, we have
$$
    |(A-A)\cap[0,a]|\geq 2(\F{x+1}-\F{y})-t.
$$
Since $A$ is sum free, $(A-A)\cap[0,a]$ is disjoint from $A\cap[0,a]$ up to the endpoint $0$, which has measure zero. Therefore
$$
    2(\F{x+1}-\F{y})-t+\F{a}\leq a,
$$
and hence
$$
    2(\F{x+1}-\F{y})+\F{a}\leq a+t.
$$
Also, by the maximality of the discrepancy of $I$, applied to $[a,x-1]$, we have
$$
\F{x-1}-\F{a}\leq \frac{x-1-a+t}{2}.
$$
Hence,  
$$
2\F{x+1}+\F{x-1}\leq a+t+2\F{y}+\frac{x-1-a+t}{2}
$$
which, by the previous inequalities and the definition of $y$ implies
$$
2\F{x+1}+\F{x-1}\leq x+s+2\F{u}-\frac{y-t}{2}.
$$
Using $\F{u}\leq y-2-(s+t)/2$, we obtain
$$
2\F{x+1}+\F{x-1}
\leq x-4+\frac{3y-t}{2}.
$$
Since $\F{x+1}-\F{x}>1-t$, we have
$$
\F{x}<\F{x+1}-1+t.
$$
Therefore
$$
\F{x-1}+\F{x}+\F{x+1}
<x-5+\frac{3y+t}{2}.
$$
But $t\leq s<y-2$ and $y\leq3$, so $3y+t<4y-2\leq10$. Thus
$$
\F{x-1}+\F{x}+\F{x+1}<x,
$$
contradicting the choice of $x$.

We may therefore assume that $A\cap[b+1,x]=\varnothing$. In particular, $\F{x}=\F{b+1}$. Repeating the argument used in the proof of the preceding lemma up to that point, and using only the fact that $y<3+s$, we get
$$
\F{x-1}+2\F{b+1}
\leq
x-\frac{y+1}{2}+m_A^-(u)+s+2(\F{a+1}-\F{b}).
$$
Since
$$
(A\cap[0,y-1])+(A\cap[x-1,a+1])\subseteq [x,x+1],
$$
and this sumset is disjoint from $A$, we have
$$
|A\cap[x,x+1]|+
|(A\cap[0,y-1])+(A\cap[x-1,a+1])|
\leq 1.
$$
By Brunn--Minkowski,
$$
|(A\cap[0,y-1])+(A\cap[x-1,a+1])|
\geq \F{y-1}+\F{a+1}-\F{x-1}.
$$
Thus
$$
(\F{x+1}-\F{x})+\F{y-1}+(\F{a+1}-\F{x-1})\leq 1.
$$
Combining this with the previous estimate and using $\F{x}=\F{b+1}$ gives
$$
\F{x-1}+\F{x}+\F{x+1}
\leq
x-\frac{y+1}{2}+s+(\F{a+1}-\F{b})+(\F{x-1}-\F{b})+1,
$$
where we used $m_A^-(u)\leq \F{u}=\F{y-1}$.
It remains to bound the two remaining terms. Since $I=[a,b]$ is discrepancy maximizing and
$|A\cap I|=(s+t)/2$, comparison with the interval $[a,a+1]$ gives
$$
2(\F{a+1}-\F{a})-1\leq t.
$$
Subtracting $2(\F{b}-\F{a})-s=t$, we get
$$
\F{a+1}-\F{b}\leq \frac{1-s}{2}.
$$
Similarly, comparison with $[a,x-1]$ gives
$$
2(\F{x-1}-\F{a})-(x-1-a)\leq t.
$$
Subtracting again $2(\F{b}-\F{a})-s=t$, we get
$$
\F{x-1}-\F{b}\leq \frac{x-1-b}{2}.
$$
Therefore
$$
\F{x-1}+\F{x}+\F{x+1}
\leq
x-\frac{y+1}{2}+s+\frac{1-s}{2}+\frac{x-1-b}{2}+1
=
x,
$$
because $x-1-b=y-s-2$. This gives a contradiction.
\end{proof}
\subsection{A gap around $x$ and the inequality $t>1/2$}
We begin by proving that the left end-point $a$ of the discrepancy-maximizing interval cannot be too close to $1$. This will lead to a gap of length at least $1$ around $x$, and hence to the inequality $t>1/2$.
\begin{lem}
   $a>2$.
\end{lem}
\begin{proof}
    Assume that $a\leq 2$. We first observe that either $a\geq v$ or $b\leq u$. Indeed, suppose that $a<v$. Since $a\in A$, the definition of $v$ gives $a<y-1$, and hence $a\leq u$. If $b>u$, then $b\in A$ and $(u,v)\cap A=\varnothing$ force $b\geq v$, so $[u,v]\subseteq I$. But $v-u\geq1$ and $|I|=s\leq1$, so this forces $I=[u,v]$. Since $(u,v)\cap A=\varnothing$, this gives $\disc_A(I)\leq0$, contradicting $t=\disc_A(I)>0$. If $a\geq v$, then, since $v\geq y-1>2$, we immediately get $a>2$, contradicting $a\leq2$. Therefore, $b\leq u\leq y - 1$. In particular, $y > a$. By Lemma \ref{mzero}, we have $m_A^-(y)=0$, since $a+y=x+1\in A$.
Therefore, by Lemma \ref{minterv}, applied with
$[\alpha,\beta]=[a,y]$, we get
$$
    |(A+A)\cap[2a,x+1]|\geq 2(\F{y}-\F{a}).
$$
Since $A$ is sum free we have
$$
    2(\F{y}-\F{a})+\F{x+1}-\F{2a}\leq x+1-2a=y-a.
$$
Note that $y = x + 1 - a\in [x - 1, x]$. Since $x+1\in A$, the reflection of $A\cap[1,a]$ through $x+1$ is disjoint
from $A\cap[y,x]$. Indeed, if $r\in A\cap[1,a]$, then $x+1-r\in[y,x]$, and
$x+1-r$ cannot belong to $A$, since otherwise
$r+(x+1-r)=x+1\in A$. Therefore
\begin{equation*}
    (\F{x} - \F{y}) + \F{a} \leq x-y=a - 1. 
\end{equation*}
Therefore, 
\begin{equation*}
    \F{x - 1} + \F{x} + \F{x + 1}
    \leq y - 1 + \F{a} + \F{2a}
    = x - a + \F{a} + \F{2a}. 
\end{equation*}
Since $a\leq 2$, we have $2a\leq a+2$, and therefore, by sum-freeness and $1\in A$, we have
\begin{equation*}
    \F{2a}-\F{a}\leq 1.
\end{equation*}
Hence,  
\begin{equation*}
    \F{x - 1} + \F{x} + \F{x + 1}
    \leq x-a+1+2\F{a}.
\end{equation*}
Finally, by maximality of $I$, we have $\F{a}\leq (a-1)/2$, otherwise $[1,b]$
would have larger discrepancy than $I$. Therefore
\begin{equation*}
    \F{x - 1} + \F{x} + \F{x + 1}
    \leq x-a+1+a-1=x,
\end{equation*}
a contradiction.
\end{proof}
Since $x-1-y=a-2$, Theorem \ref{offdiagonal main2}, applied to $A\cap[y,x-1]$, gives
$$
    |(A-A)\cap[0,a-2]|
    \geq 2(\F{x-1}-\F{y})-t,
$$
where we used $d(A\cap[y,x-1])\leq t$. Since $A$ is sum free, $(A-A)\cap[0,a-2]$ is disjoint from $A\cap[0,a-2]$. Therefore
\begin{equation} \label{diff eq}
    2(\F{x-1}-\F{y})+\F{a-2}\leq a-2+t.
\end{equation}
Similarly, since $x+1-y=a$, Theorem \ref{offdiagonal main2}, applied to $A\cap[y,x+1]$, gives
$$
    |(A-A)\cap[0,a]|
    \geq 2(\F{x+1}-\F{y})-t.
$$
Again by sum-freeness, $(A-A)\cap[0,a]$ is disjoint from $A\cap[0,a]$, and hence,  
\begin{equation} \label{alt diff eq}
    2(\F{x+1}-\F{y})+\F{a}\leq a+t.
\end{equation}
By Lemma \ref{mzero}, applied with $z=y-2$, we have $m_A^-(y-2)=0$, since
$$
    a+(y-2)=x-1\in A.
$$
Therefore, by Lemma \ref{minterv}, applied to
$[\alpha,\beta]=[1,y-2]$, we get
$$
    |(A+A)\cap[a+1,a+y-2]|
    \geq 2|A\cap[1,y-2]|
    =
    2\F{y-2}.
$$
Since this sumset is contained in $[a+1,a+y-2]=[a+1,x-1]$ and is disjoint from $A$, we get
$$
    2\F{y-2}+(\F{x-1}-\F{a+1})\leq y-3.
$$
We define $g\geq0$ by
\begin{equation} \label{direct disc eq}
    2\F{y-2}+(\F{x-1}-\F{a+1})=y-3-g.
\end{equation}
Combining \eqref{alt diff eq} with \eqref{direct disc eq}, and using $a+y=x+1$, gives
$$
    \F{x-1}+2\F{x+1}
    \leq x-2+t-g+(\F{a+1}-\F{a})+2(\F{y}-\F{y-2}).
$$
By the maximality of the discrepancy of $I$, applied to $[a,a+1]$, we have
$$
    \F{a+1}-\F{a}\leq \frac{1+t}{2}.
$$
Therefore
$$
    \F{x-1}+2\F{x+1}
    \leq
    x-2+t-g+\frac{1+t}{2}+2(\F{y}-\F{y-2}).
$$
Since
$$
    \F{x-1}+\F{x}+\F{x+1}>x
$$
and $\F{x+1}-\F{x}>1-t$, we have
$$
    \F{x-1}+2\F{x+1}>x+1-t.
$$
Together with the previous upper bound, this gives
$$
    \F{y}-\F{y-2}>\frac{5}{4}(1-t)+\frac{g}{2}.
$$
In particular,
$$
    \F{y}-\F{y-2}>\frac{5}{4}(1-t).
$$
From this estimate we deduce also that $u>y-2$, that will be used in
the proof of the next lemma. To prove this, we argue by contradiction. Suppose
that $u\leq y-2$. Since $u=\max(A\cap (-\infty,y-1])$, this implies that
$A\cap (y-2,y-1]=\varnothing$. Moreover, by Lemma \ref{gapy}, the set $A$ is
disjoint from $[y-1,y-s]$. Hence,  
$$
\F{y}-\F{y-2}=|A\cap [y-s,y]|.
$$
We now use that $x+1\in A$ and that $A$ is sum free. Since
$$
x+1-[y-s,y]=[a,b]=I,
$$
the map $z\mapsto x+1-z$ sends $A\cap [y-s,y]$ into $I\setminus A$. Therefore
$$
\F{y}-\F{y-2}
=
|A\cap [y-s,y]|
\leq |I\setminus A|
=
s-|A\cap I|
=
\frac{s-t}{2}.
$$
Since $s\leq 1$, this gives
$$
\F{y}-\F{y-2}\leq \frac{1-t}{2},
$$
contradicting the previous estimate. This proves that $u>y-2$.

We now return to the neighbourhood of $x$. The next lemma shows that the information already obtained forces another empty interval there.
\begin{lem}
    $A$ is disjoint from $[u+b-1,x]$, which contains $[x+s-1,x]$
\end{lem}
\begin{proof}
Since $u\leq y-1$, we have
$$
    u+b-1\leq y-1+b-1=x+s-1,
$$
so $[u+b-1,x]$ contains $[x+s-1,x]$.

Suppose, for a contradiction, that there exists
$z\in A\cap[u+b-1,x]$. Then
$$
    z-(A\cap I)\subseteq [z-b,z-a]
$$
and this set is disjoint from $A$, by sum-freeness.

We claim that there is an interval $I'$ of length $1$ containing both
$[z-b,z-a]$ and $[y-2,u]$. Indeed, since $z\in[u+b-1,x]$, we have
$$
    z-b\geq u-1,
    \qquad
    z-a\leq x-a=y-1.
$$
If $z-b\geq y-2$, then $[z-b,z-a]\subseteq[y-2,y-1]$, and also
$[y-2,u]\subseteq[y-2,y-1]$. If instead $z-b<y-2$, then
$[z-b,z-a]\subseteq[z-b,z-b+1]$, because $s\leq1$, and
$[y-2,u]\subseteq[z-b,z-b+1]$, because $z-b<y-2$ and $z-b\geq u-1$.
This proves the claim.

Now $z-(A\cap I)$ and $A\cap[y-2,u]$ are disjoint subsets of the interval
$I'$, since the first one is disjoint from $A$ and the second one is contained
in $A$. Therefore
$$
    \frac{s+t}{2}+\F{u}-\F{y-2}\leq 1,
$$
and hence
$$
    \F{u}-\F{y-2}\leq 1-\frac{s+t}{2}.
$$

On the other hand, reflection through $x+1$ sends $[y-s,y]$ onto $I$.
Since $x+1\in A$, this reflection sends $A\cap[y-s,y]$ into $I\setminus A$.
Thus
$$
    \F{y}-\F{y-s}\leq |I\setminus A|=\frac{s-t}{2}.
$$
Since $A\cap[y-1,y-s]=\varnothing$, we have
$$
    \F{y}-\F{y-1}=\F{y}-\F{y-s}\leq \frac{s-t}{2}.
$$
Also, by the definition of $u$, there is no mass of $A$ in $(u,y-1]$, so
$\F{y-1}=\F{u}$. Therefore
$$
\begin{aligned}
    \F{y}-\F{y-2}
    &=
    (\F{u}-\F{y-2})+(\F{y}-\F{y-1})  \\
    &\leq
    1-\frac{s+t}{2}+\frac{s-t}{2}
    =
    1-t,
\end{aligned}
$$
contradicting the previously proved inequality
$$
    \F{y}-\F{y-2}>\frac{5}{4}(1-t).
$$
\end{proof}
We now show that there is a gap of length at least $1$ around $x$. To this end, let
$$
u':=\max(A\cap(-\infty,x]),\qquad v':=\min(A\cap[x,\infty)).
$$
Thus it will be enough to prove that $v'-u'\geq 1$.
\begin{lem}
    $v'\geq u' + 1$. 
\end{lem}
\begin{proof}
Suppose, for a contradiction, that $v'-u'<1$. By the previous lemma, $A$ is
disjoint from $[u+b-1,x]$. Since $x-1\in A$, this gives
$$
    u'\in [x-1,u+b-1).
$$
Moreover, $x\notin A$ and $x+1\in A$, so $v'>x$; hence, using
$v'-u'<1$, we also have
$$
    v'\in (x,u+b).
$$

Since $u',v'\in A$, both sets
$$
    u'-(A\cap I),
    \qquad
    v'-(A\cap I)
$$
are disjoint from $A$. Therefore, inside $[y-2,u]$,
$$
    \left| \bigl((u'-(A\cap I))\cup(v'-(A\cap I))\bigr)\cap[y-2,u]\right|
    +\F{y-1}-\F{y-2}
    \leq u-y+2.
$$
Here we used that $\F{y-1}=\F{u}$, since $A\cap(u,y-1]=\varnothing$.

Passing from $A\cap I$ to the full interval $I$, the first translate can lose
at most $|I\setminus A|=(s-t)/2$, while the second translate loses at most
$|A^c\cap[v'-u,b]|$. Hence,  
$$
\begin{aligned}
    &\left| \bigl((u'-I)\cup(v'-I)\bigr)\cap[y-2,u]\right|
    +\F{y-1}-\F{y-2} \\
    &\qquad\leq
    u-y+2+\frac{s-t}{2}+|A^c\cap[v'-u,b]|.
\end{aligned}
$$
If $v\geq y$, then $A\cap[y-1,y]=\varnothing$, and hence $\F{y}=\F{y-1}$. Moreover, $[v'-u,b]\subseteq I$. Indeed, since $v'\geq x$ and
$u\leq y-1$, we have
$$
    v'-u\geq x-u\geq x-y+1=a.
$$
Moreover, since we are assuming $v'-u'<1$ and we have just proved
$u'<u+b-1$, we get the other inclusion since
$$
    v'<u'+1<u+b.
$$
Therefore
$$
    |A^c\cap [v'-u,b]|
    \leq |I\setminus A|
    =
    \frac{s-t}{2}.
$$
Thus, in this case,
$$
\begin{aligned}
    &\left|
        \bigl((u'-I)\cup(v'-I)\bigr)\cap [y-2,u]
    \right|
    +\F{y}-\F{y-2}  \\
    &\qquad\leq
    u-y+2+\frac{s-t}{2}+|A^c\cap [v'-u,b]|  \\
    &\qquad\leq
    u-y+2+s-t.
\end{aligned}
$$

If $v<y$ we have, by definition of $v$,
$$
    \F{y}-\F{y-1}=\F{y}-\F{v}\leq |A^c\cap[a,a+y-v]|.
$$
Indeed, reflection through $x+1$ sends $A\cap[v,y]$ into
$[a,a+y-v]\setminus A$.

Combining the last two estimates gives
$$
\begin{aligned}
    &\left| \bigl((u'-I)\cup(v'-I)\bigr)\cap[y-2,u]\right|
    +\F{y}-\F{y-2} \\
    &\qquad\leq
    u-y+2+\frac{s-t}{2}
    +|A^c\cap[v'-u,b]|+|A^c\cap[a,a+y-v]|.
\end{aligned}
$$
Since $v-u\geq1$ and $v'\geq x$, we have
$$
    a+y-v=x+1-v\leq v'-u.
$$
Thus the intervals $[a,a+y-v]$ and $[v'-u,b]$ are disjoint subintervals of
$I$. Therefore
$$
    |A^c\cap[v'-u,b]|+|A^c\cap[a,a+y-v]|
    \leq |I\setminus A|
    =
    \frac{s-t}{2}.
$$
Hence, as in the previous case we arrive at
$$
    \left| \bigl((u'-I)\cup(v'-I)\bigr)\cap[y-2,u]\right|
    +\F{y}-\F{y-2}
    \leq u-y+2+s-t.
$$
Using the previously proved inequality
$$
    \F{y}-\F{y-2}>\frac54(1-t),
$$
we obtain
$$
    \left| \bigl((u'-I)\cup(v'-I)\bigr)\cap[y-2,u]\right|
    < u-y+1+s.
$$

On the other hand, $u'-I$ contains $y-2$, because $u'\geq x-1=a+y-2$ and
$u'<u+b-1$. Also $v'-I$ contains $u$, because $v'>x$ and $v'<u+b$.
Finally,
$$
    \min(v'-I)-\max(u'-I)
    =
    (v'-b)-(u'-a)
    =
    v'-u'-s
    <1-s.
$$
Therefore,   the union $(u'-I)\cup(v'-I)$ covers all of $[y-2,u]$ except for
an interval of length strictly less than $1-s$. Hence,  
$$
    \left| \bigl((u'-I)\cup(v'-I)\bigr)\cap[y-2,u]\right|
    >
    u-y+2-(1-s)
    =
    u-y+1+s,
$$
contradicting the previous upper bound.
\end{proof}

The gap around $x$ has an immediate consequence for the discrepancy-maximizing interval itself. We now show that its discrepancy must in fact be greater than $1/2$.
\begin{cor} \label{rough t bound}
    $t > \frac{1}{2}$. 
\end{cor}

\begin{proof}
Since $\F{x}=\F{u'}$, by Theorem \ref{offdiagonal main2} applied to
$A\cap[1,u']$, and using $d(A\cap[1,u'])\leq t$, we have
$$
    |(A-A)\cap[0,u'-1]|\geq 2\F{u'}-t.
$$
Since $A$ is sum free, $(A-A)\cap[0,u'-1]$ is disjoint from
$A\cap[0,u'-1]$. Therefore
$$
    2\F{x}+\F{u'-1}=2\F{u'}+\F{u'-1}\leq u'-1+t.
$$
Moreover, since $\F{x}=\F{v'}$, maximality of $I$ applied to $[v',x+1]$ gives
$$
    \F{x+1}-\F{x}
    =
    \F{x+1}-\F{v'}
    \leq \frac{x+1-v'+t}{2}.
$$
Since $v'-u'\geq1$, this implies
$$
    \F{x+1}-\F{x}
    \leq \frac{x-u'+t}{2}.
$$
Similarly, maximality of $I$ applied to $[u'-1,x-1]$ gives
$$
    \F{x-1}-\F{u'-1}
    \leq \frac{x-u'+t}{2}.
$$
Adding the three inequalities, we obtain
$$
    \F{x-1}+\F{x}+\F{x+1}
    \leq x+2t-1.
$$
Since $\F{x-1}+\F{x}+\F{x+1}>x$, it follows that $t>\frac12$.
\end{proof}
\subsection{Reducing to the range $3<y<3+s$}
In this subsection we show that $y$ must in fact lie in the range $3<y<3+s$. The lower bound $y>3$ has already been established, so the remaining goal is to prove the upper bound $y<3+s$, using the gap around $x$ and the inequality $t>1/2$.

Combining \eqref{diff eq} and \eqref{direct disc eq}, and using
$a+y=x+1$, gives
$$
    3\F{x - 1}
    \leq
    x - 4 + t
    +2(\F{y}-\F{y-2})
    +\F{a+1}-\F{a-2}
    -g .
$$
Since $A\cap[y-1,y-s]=\varnothing$, we have
$$
    |A\cap[y-1,y]|=|A\cap[y-s,y]|.
$$
Moreover, reflection through $x+1$ sends $A\cap[y-s,y]$ into
$I\setminus A$, because $x+1\in A$ and $A$ is sum free. Hence,  
$$
    \F{y}-\F{y-1}
    =
    |A\cap[y-1,y]|
    \leq
    |I\setminus A|
    =
    \frac{s-t}{2}.
$$
Therefore,   we have
\begin{equation} \label{three Fxminus estimate}
    3\F{x - 1}
    \leq
    x - 4 + s
    +2(\F{y-1}-\F{y-2})
    +\F{a+1}-\F{a-2}
    -g .
\end{equation}
We now estimate the mass of $A$ inside $[x-1,x+1]$. Since
$v'\geq u'+1$ and $s\leq1$, the interval $[u',u'+s]$ is disjoint from $A$,
up to endpoints.

We shall also construct a further subset of $[x-1,x+1]\setminus A$ coming
from $A+A$. First,
$$
    (A\cap[y-2,u'-a])+a\subseteq [x-1,u'].
$$
Second,
$$
    (A\cap[u'-a,u])+(A\cap[b,a+1])\subseteq [u'+s,u+a+1].
$$
Indeed, the left endpoint is $u'-a+b=u'+s$, while the right endpoint is
$u+a+1$. Finally, if $v\leq y$, then
$$
    (A\cap[v,y])+a\subseteq [v+a,x+1].
$$
All these sets are contained in $A+A$, and hence are disjoint from $A$, since
$A$ is sum free.

The three intervals just considered are disjoint up to endpoints. Indeed, the
first ends at $u'$, the second starts at $u'+s$, and the third starts at
$v+a\geq u+a+1$, because $v-u\geq1$. Therefore,   the corresponding sumset pieces
are disjoint up to endpoints.

The first piece has measure
$$
    \F{u'-a}-\F{y-2}.
$$
For the second piece, by the one-dimensional Brunn--Minkowski inequality,
$$
    |(A\cap[u'-a,u])+(A\cap[b,a+1])|
    \geq
    (\F{u}-\F{u'-a})+(\F{a+1}-\F{b}).
$$
If $v\leq y$, the third piece has measure
$$
    \F{y}-\F{v}.
$$
Since $A\cap(u,v)=\varnothing$, we have $\F{v}=\F{u}$. Hence, in this case, the
total forbidden contribution coming from these sumsets is at least
$$
    (\F{u'-a}-\F{y-2})
    +(\F{u}-\F{u'-a})
    +(\F{a+1}-\F{b})
    +(\F{y}-\F{v})
    =
    \F{y}-\F{y-2}+\F{a+1}-\F{b}.
$$
If instead $v>y$, then $\F{y}=\F{u}$, and the same lower bound follows from the
first two pieces alone.

Thus, inside the interval $[x-1,x+1]$, the set $A$ misses the interval
$[u',u'+s]$, of length $s$, and also misses a further disjoint set of measure
at least
$$
    \F{y}-\F{y-2}+\F{a+1}-\F{b}.
$$
Since $[x-1,x+1]$ has length $2$, we obtain
\begin{equation} \label{x interval mass estimate}
    \F{x+1}-\F{x-1}
    \leq
    2-s
    -(\F{y}-\F{y-2})
    -(\F{a+1}-\F{b}).
\end{equation}
Combining \eqref{three Fxminus estimate} with \eqref{x interval mass estimate}, we obtain, by using $\F{y-1}\leq \F{y}$ and $\F{b}\leq \F{a+1}$
$$
    2\F{x+1}+\F{x-1}
    \leq
    x-s+\F{b}-\F{a-2}-g.
$$
Since $\F{b}-\F{a}=|A\cap I|=(s+t)/2$, this becomes
$$
    2\F{x+1}+\F{x-1}
    \leq
    x-\frac{s-t}{2}+\F{a}-\F{a-2}-g.
$$
On the other hand, our standing contradiction assumption, together with
$\F{x+1}-\F{x}>1-t$, gives
$$
    2\F{x+1}+\F{x-1}
    =
    \F{x-1}+\F{x}+\F{x+1}+(\F{x+1}-\F{x})
    >
    x+1-t.
$$
Comparing the last two estimates yields
\begin{equation} \label{boost and interval around a}
    1-t+g+\frac{s-t}{2}<\F{a}-\F{a-2}.
\end{equation}

At this point, \eqref{boost and interval around a} shows that there is already a substantial amount of mass in the interval $[a-2,a]$. The next lemma identifies the part of this interval that will matter later, namely $[a-2,a-1]$, and will allow us to define a parameter $w$ from the leftmost point of $A$ there. In the extremal period-3 picture, one should think of $[a-2,a-1]$ as belonging to the preceding gap, so $w$ measures how far $A$ reaches into a region that would otherwise be empty.
\begin{lem}
    $A\cap[a-2,a-1]\neq \varnothing$.
\end{lem}
\begin{proof}
    Suppose not. Then
$$
    \F{a}-\F{a-2}=\F{a}-\F{a-1}.
$$
Since $1\in A$ and $A$ is sum free, the translate
$$
    1+(A\cap[a-1,b-1])
$$
is contained in $[a,b]$ and is disjoint from $A$. Hence,  
$$
    |A\cap[a-1,b-1]|+|A\cap[a,b]|\leq s.
$$
Since $|A\cap[a,b]|=(s+t)/2$, this gives
$$
    |A\cap[a-1,b-1]|\leq \frac{s-t}{2}.
$$
Moreover, since $b=a+s$ and $s\leq1$, we have $b-1\leq a$, and therefore
$$
    |A\cap[b-1,a]|\leq a-(b-1)=1-s.
$$
Thus
$$
    \F{a}-\F{a-1}
    \leq
    \frac{s-t}{2}+1-s
    =
    1-\frac{s+t}{2}.
$$
On the other hand, \eqref{boost and interval around a} gives
$$
    \F{a}-\F{a-2}
    >
    1-t+g+\frac{s-t}{2}
    =
    1-\frac{s+t}{2}+(s-t)+g
    \geq
    1-\frac{s+t}{2},
$$
a contradiction. Therefore,   $A\cap[a-2,a-1]\neq\varnothing$.
\end{proof}
Thanks to the previous lemma we can define $w$ so that $a - 2 + w = \max(A\cap [a - 2, a - 1])$. Then $a-2+w\in A$ and $A\cap(a-2+w,a-1]=\varnothing$, so $\F{a-1}=\F{a-2+w}$.
Hence,
$$
    \F{a}-\F{a-2}
    =
    (\F{a-2+w}-\F{a-2})
    +(\F{a-1+w}-\F{a-1})
    +(\F{a}-\F{a-1+w}).
$$
Since $1\in A$ and $A$ is sum free, the translate
$$
    1+(A\cap[a-2,a-2+w])
$$
is contained in $[a-1,a-1+w]$ and is disjoint from $A$. Therefore,
$$
    (\F{a-2+w}-\F{a-2})+(\F{a-1+w}-\F{a-1})\leq w.
$$
Moreover, since $a-2+w\in A$, the translate
$$
    (a-2+w)+(A\cap[1,2-w])
$$
is contained in $[a-1+w,a]$ and is disjoint from $A$. Hence,
$$
    \F{a}-\F{a-1+w}\leq 1-w-\F{2-w}.
$$
Combining these estimates gives
\begin{equation} \label{w ineq}
    \F{a}-\F{a-2}\leq 1-\F{2-w}.
\end{equation}

The next lemma gives a useful lower bound for the quantity $g$. For any point $c\in A$ with $c\leq y-s-2$, we combine the sum-set contributions coming from $[1,c]$ and $[c+s,y-2]$ with the translated set
$
c+(A\cap [a+m_A^-(c),b]).
$
This produces an extra contribution inside $[a+1,x-1]$, with the only loss coming from the interval $[c,c+s]$.

\begin{lem} \label{c lem}
    For any $c\in A$, if $c\leq y - s - 2 = x - b - 1$, then
    \begin{equation*}
        g\geq |A\cap [a + m_A^-(c), b]| - 2|A\cap [c, c + s]|. 
    \end{equation*}
\end{lem}

\begin{proof}
    Since $I$ is discrepancy maximizing, we have
    $$
        0\leq m_A^-(c)\leq t\leq s.
    $$
    By Lemma \ref{minterv}, applied to $[1,c]$, we have
    $$
        |(A+A)\cap [a+1,a+c+m_A^-(c)]|\geq 2\F{c}.
    $$
    Also, since $c\leq y-s-2$, we have $c+s\leq y-2$. Moreover
    $m_A^-(y-2)=0$, because
    $$
        a+(y-2)=x-1\in A.
    $$
    Hence, by Lemma \ref{minterv}, applied to
    $[c+s,y-2]$, we get
    $$
        |(A+A)\cap [a+c+s,a+y-2]|
        \geq 2(\F{y-2}-\F{c+s}).
    $$
     Since $a+c+s=b+c$ and $a+y-2=x-1$, this becomes
    $$
        |(A+A)\cap [b+c,x-1]|
        \geq 2(\F{y-2}-\F{c+s}).
    $$
    Finally, since $c\in A$, we have
    $$
        c+(A\cap [a+m_A^-(c),b])\subseteq A+A.
    $$
    This set is contained in $[a+c+m_A^-(c),b+c]$ and has measure
    $$
        |A\cap [a+m_A^-(c),b]|.
    $$

    The intervals
    $$
        [a+1,a+c+m_A^-(c)],\qquad
        [a+c+m_A^-(c),b+c],\qquad
        [b+c,x-1]
    $$
    are disjoint up to endpoints. Indeed, $m_A^-(c)\leq s$ gives
    $a+c+m_A^-(c)\leq a+c+s=b+c$, and the assumption
    $c\leq x-b-1$ gives $b+c\leq x-1$.
    Therefore, inside $[a+1,x-1]$, the set $A+A$ has measure at least
    $$
        2\F{c}
        +|A\cap [a+m_A^-(c),b]|
        +2(\F{y-2}-\F{c+s}).
    $$
    Since $A$ is sum free, we have
    $$
        \F{x-1}-\F{a+1}
        +2\F{c}
        +|A\cap [a+m_A^-(c),b]|
        +2(\F{y-2}-\F{c+s})
        \leq y-3.
    $$
    Using
    $$
        2\F{y-2}+\F{x-1}-\F{a+1}=y-3-g,
    $$
    we obtain
    $$
        2\F{c}
        +|A\cap [a+m_A^-(c),b]|
        +2(\F{y-2}-\F{c+s})
        \leq 2\F{y-2}+g.
    $$
    Rearranging gives the desired inequality
    $$
        g\geq |A\cap [a+m_A^-(c),b]|-2(\F{c+s}-\F{c}).
    $$
\end{proof}

The next lemma shows that the parameter $w$ gives a corresponding improvement in the upper bound for $y$. The more $A$ extends into the interval $[a-2,a-1]$, the stronger the resulting bound on $y$.
\begin{lem}
    $y < 4 - w + s$. 
\end{lem}

\begin{proof}
    Suppose, for a contradiction, that $y\geq 4-w+s$. Since $0\leq w\leq 1$
    and $1\in A$, the set $A\cap[1,2-w]$ is non-empty. Let $ c=\max(A\cap[1,2-w])$. Then $c\in A$ and $c\leq 2-w$. Moreover, our contradiction assumption gives $ 2-w\leq y-s-2, $
    and hence $c\leq y-s-2$. Therefore,   Lemma \ref{c lem} gives
    $$
        g\geq |A\cap [a+m_A^-(c),b]|-2|A\cap[c,c+s]|.
    $$
    We first estimate the second term. By the maximality of $c$, we have $A\cap(c,2-w]=\varnothing$.
    If $c+s<2-w$, then $(c,c+s]\cap A=\emptyset$, by the maximality of $c$, so the term $|A\cap[c,c+s]|$ is zero and the argument only becomes easier. We may
therefore assume that $c+s\geq 2-w$. Hence, up to endpoints,
$$
    A\cap[c,c+s]\subseteq A\cap[2-w,c+s].
$$
     Since $a-2+w\in A$, the translate
    $$
        (a-2+w)+(A\cap[2-w,c+s])
    $$
    is disjoint from $A$. Moreover, since $c\leq 2-w$, this translate is
    contained in $I=[a,b]$, because
    $$
        (a-2+w)+(2-w)=a
    $$
    and
    $$
        (a-2+w)+(c+s)\leq a+s=b.
    $$
    Hence,  
    $$
        |A\cap[c,c+s]|
        \leq
        |I\setminus A|
        =
        \frac{s-t}{2}.
    $$
    Thus Lemma \ref{c lem} gives
    $$
        g\geq |A\cap [a+m_A^-(c),b]|-(s-t).
    $$
    We now estimate the first term. Since passing from $[a,b]$ to
    $[a+m_A^-(c),b]$ removes an interval of length $m_A^-(c)$, we have
    $$
        |A\cap [a+m_A^-(c),b]|
        \geq
        |A\cap I|-m_A^-(c).
    $$
    Also $m_A^-(c)\leq \F{c}$. Therefore
    $$
        |A\cap [a+m_A^-(c),b]|
        \geq
        \frac{s+t}{2}-\F{c}.
    $$
    Combining the last two estimates, we get
    $$
        g\geq \frac{s+t}{2}-\F{c}-(s-t).
    $$

    Now \eqref{boost and interval around a} gives
    $$
        1-t+g+\frac{s-t}{2}<\F{a}-\F{a-2}.
    $$
    Substituting the lower bound for $g$, we obtain
    $$
        \F{a}-\F{a-2}
        >
        1-t+\frac{s-t}{2}+\frac{s+t}{2}-\F{c}-(s-t)
        =
        1-\F{c}.
    $$
    By the maximality of $c$, we have $\F{c}=\F{2-w}$, and hence
    $$
        \F{a}-\F{a-2}>1-\F{2-w},
    $$
    contradicting \eqref{w ineq}.
\end{proof}
We finally prove that $y< 3+s$ by using the two previous lemmas.
\begin{lem}
    $y < 3 + s$. 
\end{lem}

\begin{proof}
Suppose, for a contradiction, that $y\geq 3+s$. Since $1\in A$ and
    $y-2-s\geq1$, we may define
    $$
        c=\max(A\cap[1,y-2-s]).
    $$
    Then $c\in A$ and $c\leq y-2-s=y-s-2$, so Lemma \ref{c lem} gives
    $$
        g\geq |A\cap[a+m_A^-(c),b]|-2|A\cap[c,c+s]|.
    $$

    By the maximality of $c$, we have $A\cap(c,y-2-s]=\varnothing$.
If $c+s<y-2-s$, then $(c,c+s]\cap A=\emptyset$, by the maximality of $c$, so the
term $|A\cap[c,c+s]|$ is zero and the argument only becomes easier. We may
therefore assume that $c+s\geq y-2-s$. Hence, up to endpoints,
$$
    A\cap[c,c+s]\subseteq A\cap[y-2-s,y-2].
$$
    Reflection through $x-1$ sends $[y-2-s,y-2]$ onto $I=[a,b]$. Since
    $x-1\in A$ and $A$ is sum free, it sends $A\cap[y-2-s,y-2]$ into
    $I\setminus A$. Therefore
    $$
        |A\cap[c,c+s]|\leq |I\setminus A|=\frac{s-t}{2}.
    $$
    Hence,  
    $$
        g\geq |A\cap[a+m_A^-(c),b]|-(s-t).
    $$
    Since passing from $[a,b]$ to $[a+m_A^-(c),b]$ removes an interval of
    length $m_A^-(c)$, we have
    $$
        |A\cap[a+m_A^-(c),b]|
        \geq
        \frac{s+t}{2}-m_A^-(c).
    $$
    Thus
    $$
        g\geq \frac{s+t}{2}-m_A^-(c)-(s-t).
    $$
 Combining this with \eqref{boost and interval around a}, we get
    $$
        \F{a}-\F{a-2}
        >
        1-t+\frac{s-t}{2}
        +\frac{s+t}{2}
        -m_A^-(c)
        -(s-t)
        =
        1-m_A^-(c).
    $$
    Since $m_A^-(c)\leq \F{c}$, it follows that
    $$
        \F{a}-\F{a-2}>1-\F{c}.
    $$
    On the other hand, by \eqref{w ineq},
    $$
        \F{a}-\F{a-2}\leq 1-\F{2-w}.
    $$
    Therefore,   $\F{c}>\F{2-w}$, and hence $c>2-w$.

    Since $c\leq y-2-s$, this gives
    $$
        y-2-s>2-w,
    $$
    or equivalently
    $$
        y>4-w+s.
    $$
    This contradicts the previous lemma.
\end{proof}
\subsection{The gap around $2$}

We now transfer the information obtained near the discrepancy-maximizing interval to the region around $2$. The goal is to show that this forces a substantial gap there as well.

Since $1\in A$ and $A$ is sum free, the translate
$$
    1+(A\cap[a-1,b-1])
$$
is contained in $[a,b]$ and is disjoint from $A$. Hence,  
$$
    \F{b-1}-\F{a-1}\leq |I\setminus A|=\frac{s-t}{2}.
$$
Moreover, since $b=a+s$ and $s<1$, we have
$$
    b-2<a-1<b-1<a.
$$
Thus
$$
    \F{a}-\F{b-2}
    =
    (\F{a-1}-\F{b-2})
    +(\F{b-1}-\F{a-1})
    +(\F{a}-\F{b-1}).
$$
Again by sum-freeness, the translate
$$
    1+(A\cap[b-2,a-1])
$$
is contained in $[b-1,a]$ and is disjoint from $A$. Therefore
$$
    (\F{a-1}-\F{b-2})+(\F{a}-\F{b-1})
    \leq a-(b-1)=1-s.
$$
This last inequality, combined with the one bounding the amount of mass of $A$ in $[a-1,b-1]$ proved above, gives
$$
    \F{a}-\F{b-2}
    \leq
    1-s+\frac{s-t}{2}.
$$
Combining this with \eqref{boost and interval around a}, we obtain
$$
    \F{b-2}-\F{a-2}
    =
    \F{a}-\F{a-2}-(\F{a}-\F{b-2})
    >
    g+s-t.
$$
We now introduce the nearest points of $A$ to $2$ on either side. Let
$$
u''=\max(A\cap(-\infty,2]),\qquad v''=\min(A\cap[2,\infty)).
$$
Thus $u''$ and $v''$ are the last point of $A$ before $2$ and the first point of $A$ after $2$, respectively. The next lemma gives explicit bounds for these two quantities, and hence for the size of the gap around $2$.
\begin{lem}
    $u'' < 2 - g - \frac{s - t}{2}$, $v'' > 2 + g + \frac{s - t}{2}$ and $v'' - u''\geq s + g$. 
\end{lem}
\begin{proof}
    Suppose first that $u''\geq 2-g-\frac{s-t}{2}$. Since $u''\leq 2$ and $u''\in A$, we have
    $$
        (A\cap [a,b])-u''
        \subseteq
        \left[a-2, b-2+g+\frac{s-t}{2}\right]\cap A^c .
    $$
    Hence,  
    $$
        \F{b-2}-\F{a-2}
        \leq
        s+g+\frac{s-t}{2}-\frac{s+t}{2}
        =
        s-t+g,
    $$
    contradicting the previously proved inequality
    $$
        \F{b-2}-\F{a-2}>s-t+g.
    $$
     Similarly, suppose that $v''\leq 2+g+\frac{s-t}{2}$. Since $v''\geq 2$ and $v''\in A$, we have
    $$
        (A\cap [a,b])-v''
        \subseteq
        \left[a-2-g-\frac{s-t}{2}, b-2\right]\cap A^c .
    $$
    Hence,  
    $$
        \F{b-2}-\F{a-2}
        \leq
        s+g+\frac{s-t}{2}-\frac{s+t}{2}
        =
        s-t+g,
    $$
    again contradicting
    $$
        \F{b-2}-\F{a-2}>s-t+g.
    $$
    It remains to prove the gap estimate. Suppose, for a contradiction, that
    $$
        v''-u''<s+g.
    $$
    Then the two intervals
    $$
        [a-u'',b-u'']
        \qquad\text{and}\qquad
        [a-v'',b-v'']
    $$
    cover all of $[a-2,b-2]$ except possibly for a set of measure at most $g$. Indeed, the first interval reaches the right endpoint because $u''\leq 2$, while the second reaches the left endpoint because $v''\geq 2$, and the gap between them has length
    $$
        (a-u'')-(b-v'')=v''-u''-s<g.
    $$
    Therefore
    $$
        \F{b-2}-\F{a-2}
        \leq
        g
        + |A\cap [a-u'',b-u'']|
        + |A\cap [a-v'',b-v'']|.
    $$
    Since $u'',v''\in A$, both intervals are translates of $[a,b]$ by elements of $A$, and hence their intersections with $A$ have measure at most $|I\setminus A|=\frac{s-t}{2}$. Thus
    $$
        \F{b-2}-\F{a-2}
        \leq
        g+\frac{s-t}{2}+\frac{s-t}{2}
        =
        g+s-t,
    $$
    again contradicting
    $$
        \F{b-2}-\F{a-2}>s-t+g.
    $$
\end{proof}
\subsection{Forcing $u<2$}

We now show that in fact $u<2$. Arguing by contradiction, we assume throughout this subsection that $u>2$. Note that the case $u=2$ cannot occur, since $u\in A$ and $1\in A$, while $A$ is sum free. The next lemma is the main input: it combines the gap around $y-1$ with the gap around $2$ to give a strengthened lower bound for the portion of $A-A$ contained in $[1,b-1]$.
\begin{lem} \label{double boost lem}
    Assume that $u>2$. Then
    \begin{align*}
        |(A - A)\cap [1, b - 1]|
        &\geq 2\F{b - 1}
        + |A\cap [a, b - m_A^-(u)]|
        + |A\cap [a, b - m_A^-(u'')]|\\
        &\geq 2\F{b - 1} + s + t - m_A^-(u) - m_A^-(u'')\\
        &\geq 2\F{b - 1} + s + t - \F{u} .
    \end{align*}
\end{lem}
\begin{proof}
     Since we are in the remaining case $y<3+s$, Lemma \ref{small case} gives $a\geq v$. Since $v-u\geq1$, we get
    $$
        b-1=a+s-1\geq v+s-1\geq u+s.
    $$
    Moreover, $u>2$ implies $v''\leq u$. Hence, by the previous lemma,
    $$
        u-u''\geq v''-u''\geq s+g\geq s.
    $$
    Also,
    $$
        \F{u+s}=\F{u}
        \qquad\text{and}\qquad
        \F{u''+s}=\F{u''},
    $$
    because $A$ has no mass in the gaps $(u,u+s)$ and $(u'',u''+s)$.
Since $m_A^-(b-1)=0$, we can apply Lemma \ref{minterv}
to the interval $[u+s,b-1]$ to get
$$
    |(A-A)\cap[1,b-u-s]|
    \geq 2|A\cap[u+s,b-1]|
    =
    2(\F{b-1}-\F{u+s}).
$$
Since $u\in A$, the translate by $u$ gives
    $$
        (A\cap[a,b-m_A^-(u)])-u
        \subseteq
        (A-A)\cap[b-u-s,b-u-m_A^-(u)],
    $$
    and therefore contributes
    $$
        |A\cap[a,b-m_A^-(u)]|.
    $$
    Applying Lemma \ref{minterv} to the interval $[u''+s,u]$ we get
$$
    |(A-A)\cap[b-u-m_A^-(u),b-u''-s]|
    \geq 2|A\cap[u''+s,u]|
    =
    2(\F{u}-\F{u''+s}).
$$
Similarly, since $u''\in A$, the translate by $u''$ gives
    $$
        (A\cap[a,b-m_A^-(u'')])-u''
        \subseteq
        (A-A)\cap[b-u''-s,b-u''-m_A^-(u'')],
    $$
    and therefore contributes
    $$
        |A\cap[a,b-m_A^-(u'')]|.
    $$
        Finally, applying Lemma \ref{minterv} to the interval $[1,u'']$ we get
    $$
        |(A-A)\cap[b-u''-m_A^-(u''),b-1]|
        \geq 2|A\cap[1,u'']|
        =
        2\F{u''}.
    $$
    The intervals just constructed are disjoint, except possibly at endpoints. Indeed, using
    $m_A^-(u),m_A^-(u'')\leq t\leq s$ and $u-u''\geq s$, their endpoints are ordered as
    $$
        1\leq b-u-s\leq b-u-m_A^-(u)
        \leq b-u''-s\leq b-u''-m_A^-(u'')\leq b-1.
    $$
    Summing the contributions gives
    \begin{align*}
        |(A-A)\cap[1,b-1]|
        &\geq
        2(\F{b-1}-\F{u+s})
        + |A\cap[a,b-m_A^-(u)]|\\
        &\quad
        +2(\F{u}-\F{u''+s})
        + |A\cap[a,b-m_A^-(u'')]|
        +2\F{u''}.
    \end{align*}
    Since $\F{u+s}=\F{u}$ and $\F{u''+s}=\F{u''}$, we obtain
    $$
        |(A-A)\cap[1,b-1]|
        \geq
        2\F{b-1}
        + |A\cap[a,b-m_A^-(u)]|
        + |A\cap[a,b-m_A^-(u'')]|.
    $$

    For every $z\in\{u,u''\}$, we have
    $$
        |A\cap[a,b-m_A^-(z)]|
        \geq
        |A\cap[a,b]|-m_A^-(z)
        =
        \frac{s+t}{2}-m_A^-(z).
    $$
    Hence,
    $$
        |(A-A)\cap[1,b-1]|
        \geq
        2\F{b-1}+s+t-m_A^-(u)-m_A^-(u'').
    $$

    It remains to prove that
    $$
        m_A^-(u)+m_A^-(u'')\leq \F{u}.
    $$
    First, for every $r\leq u''$ we have
    $$
        2|A\cap[r,u'']|-(u''-r)\leq \F{u''},
    $$
    and therefore
    $$
        m_A^-(u'')\leq \F{u''}.
    $$
    We claim that
    $$
        m_A^-(u)\leq \F{u}-\F{u''}.
    $$
    Let $r\leq u$. If $r\geq u''$, then
    $$
        2|A\cap[r,u]|-(u-r)
        \leq |A\cap[r,u]|
        \leq \F{u}-\F{u''},
    $$
    where we used only $|A\cap[r,u]|\leq u-r$.

    Suppose now that $r<u''$. Since $A\cap(u'',v'')=\varnothing$, we have
    \begin{align*}
        2|A\cap[r,u]|-(u-r)
        &=
        \bigl(2|A\cap[r,u'']|-(u''-r)\bigr)\\
        &\quad
        -(v''-u'')
        +\bigl(2|A\cap[v'',u]|-(u-v'')\bigr)\\
        &\leq
        m_A^-(u'')-(v''-u'')+|A\cap[v'',u]|.
    \end{align*}
    Since
    $$
        m_A^-(u'')\leq t\leq s\leq v''-u'',
    $$
    it follows that
    $$
        2|A\cap[r,u]|-(u-r)
        \leq |A\cap[v'',u]|
        \leq \F{u}-\F{u''}.
    $$
    Taking the supremum over $r\leq u$, we get
    $$
        m_A^-(u)\leq \F{u}-\F{u''}.
    $$
which gives
    $$
        m_A^-(u)+m_A^-(u'')\leq \F{u}.
    $$
\end{proof}

\begin{lem}
    We have $u < 2$, and therefore $u = u''$ and $v = v''$. 
\end{lem}

\begin{proof}
    For the sake of contradiction, assume that $u>2$. By Lemma \ref{double boost lem}, we have
    $$
        |(A - A)\cap [1, b - 1]|
        \geq 2\F{b - 1} + s + t - \F{u}.
    $$
    Since $(A-A)\cap[1,b-1]$ is disjoint from $A\cap[1,b-1]$, we get
    $$
        2\F{b - 1} + s + t - \F{u} + \F{b-1}\leq b-2.
    $$
    Hence,  
    $$
        3\F{b - 1}\leq b - 2 - s - t + \F{u} = a - 2 - t + \F{u}.
    $$
    Since $\F{b+1}-\F{b-1}\leq1$, this implies
    $$
        3\F{b + 1}\leq a + 1 - t + \F{u}.
    $$
 Note that $A\cap[b+1,x]\neq \varnothing$.
 Suppose, for a contradiction,
that $A\cap[b+1,x]=\varnothing$. Then $\F{x}=\F{b+1}$. Since we are in the
remaining case $y<3+s$, we have
$$
    b+1=x+2-y+s>x-1,
$$
and hence $\F{x-1}\leq \F{b+1}$. Also $\F{x+1}\leq \F{b+1}+1$. Therefore, by
the contradiction assumption,
$$
    x<\F{x-1}+\F{x}+\F{x+1}\leq 3\F{b+1}+1,
$$
so
$$
    3\F{b+1}>x-1=a+y-2.
$$
On the other hand, the estimate just proved gives
$$
    3\F{b+1}\leq a+1-t+\F{u}.
$$
Consequently,
$$
    \F{u}>y-3+t.
$$
Since we are assuming $u>2$, the point $v''=\min(A\cap[2,\infty))$ satisfies
$v''\leq u$. As $A\cap(u'',v'')=\varnothing$ and $v''-u''\geq s+g$, we have
$$
    \F{u}\leq u-1-(v''-u'')\leq u-1-s-g\leq y-2-s-g,
$$
because $u\leq y-1$. Hence,  
$$
    y-3+t<\F{u}\leq y-2-s-g,
$$
which gives
$$
    s+t+g<1.
$$
This contradicts $t>1/2$ and $s\geq t$. Thus $A\cap[b+1,x]\neq\varnothing$.
By the Brunn--Minkowski inequality applied to
    $A\cap[b+1,x]$ and $A\cap[a,b]$, we have
    $$
        |(A\cap[b+1,x])-(A\cap[a,b])|
        \geq
        |A\cap[b+1,x]|+|A\cap[a,b]|.
    $$
    Moreover,
    $$
        (A\cap[b+1,x])-(A\cap[a,b])
        \subseteq [1,y-1],
    $$
    since $x-a=y-1$. This difference set is disjoint from $A\cap[1,y-1]$, by sum-freeness. Therefore
    $$
        \F{y - 1} + (\F{x} - \F{b + 1}) + (\F{b} - \F{a})\leq y - 2.
    $$
     By summing the two previous inequalities we have,
    $$
        \F{x}+2\F{b+1}
        \leq x-t+\F{u}-\F{y-1}-(\F{b}-\F{a})
        \leq x-t-\frac{s+t}{2},
    $$
    where we used $u\leq y-1$ and $\F{b}-\F{a}=(s+t)/2$. However, since $b+1\geq x-1$, we have
    $$
        \F{x - 1} + \F{x} + \F{x + 1}
        \leq x - t - \frac{s + t}{2} + 1
        \leq x,
    $$
    a contradiction.

Since $y>3$, we have $2<y-1$. Hence,   any point of $A\cap(-\infty,2]$ also lies in $A\cap(-\infty,y-1]$, so $u''\leq u$. On the other hand, $u<2$ and $u\in A$ give $u\leq u''$. Therefore,   $u=u''$. Similarly, there is no point of $A$ in $[2,y-1]$, otherwise it would contradict the definition of $u$; hence, the first point of $A$ to the right of $2$ is the first point to the right of $y-1$, and so $v=v''$.
\end{proof}

\subsection{The final contradiction}
We finally prove that under the obtained assumptions, we have a contradiction.

We first prove that
    \begin{equation}\label{eq:prep-u-small}
        \F{x-1}+\F{a-2}+\F{x+1}\leq x+t+\F{2}-3.
    \end{equation}
    We split the proof into three cases.
\medskip

\noindent\textbf{Case 1.}       First suppose that $v>y$. Since $u<2$ and $y>3$, we have
    $$
        \F{y}=\F{u}=\F{2}.
    $$
    Indeed, $2<y-1<y<v$ and $(u,v)\cap A=\varnothing$.

    By the Brunn--Minkowski inequality,
    $$
        \bigl|(A\cap[v',x+1])-(A\cap[x-1,u'])\bigr|
        \geq
        (\F{x+1}-\F{v'})+(\F{u'}-\F{x-1}).
    $$
    This difference set is contained in $[1,2]$, because $v'\geq u'+1$, and it
    is disjoint from $A$ by sum-freeness. Hence,  
    $$
        (\F{x+1}-\F{v'})+(\F{u'}-\F{x-1})+\F{2}\leq1.
    $$
    Since $A$ has no mass in $(u',v')$, we have $\F{u'}=\F{v'}$, and therefore
    $$
        \F{x+1}-\F{x-1}+\F{2}\leq1.
    $$
    Combining this with \eqref{diff eq},
    $$
        2\F{x-1}+\F{a-2}\leq a-2+t+2\F{y},
    $$
    and using $\F{y}=\F{2}$, gives
    $$
        \F{x-1}+\F{a-2}+\F{x+1}
        \leq
        a-1+t+\F{2}
        =
        x-y+t+\F{2}
        <
        x+t+\F{2}-3,
    $$
    because $y>3$. Thus \eqref{eq:prep-u-small} holds in this case.
\medskip

\noindent\textbf{Case 2.}
        Next suppose that
    $$
        v\leq y,\qquad v+a\geq v'.
    $$
    Since $u<2$, $y>3$, and $v\geq y-1$, we have
    $$
        \F{v}=\F{u}=\F{2}.
    $$
    Since in this case $v\leq y$, while $v\geq y-1$, we have
$
    v-y+2\in[1,2].
$
    Consider the set
    $$
        (A\cap[v',v+a])-(A\cap[x-1,u']).
    $$
    This difference set is contained in
    $$
        [v'-u',v+a-x+1]\subseteq[1,v-y+2],
    $$
    and is disjoint from $A$ by sum-freeness. Hence,
    $$
        (\F{v+a}-\F{v'})+(\F{u'}-\F{x-1})+\F{v-y+2}
        \leq v+a-x.
    $$
    Since $\F{u'}=\F{v'}$, this gives
    $$
        \F{v+a}-\F{x-1}+\F{v-y+2}\leq v+a-x.
    $$

    By Lemma \ref{mzero}, we have $m_A^-(y)=0$, since $a+y=x+1\in A$.
    Therefore,   Lemma \ref{minterv}, applied to $[v,y]$, gives
    $$
        |(A+A)\cap[v+a,x+1]|\geq 2(\F{y}-\F{v}).
    $$
    This sumset is disjoint from $A$, and hence
    $$
        2(\F{y}-\F{v})+(\F{x+1}-\F{v+a})\leq x+1-v-a.
    $$
    Adding the last two inequalities, and using $\F{v}=\F{2}$, gives
    $$
        \F{x+1}-\F{x-1}+\F{v-y+2}+2(\F{y}-\F{2})\leq1.
    $$

    Since $v-u\geq1$, we have
    $$
        \F{v-y+2}\geq \F{u-(y-3)}\geq \F{u}-(y-3)=\F{2}-(y-3).
    $$
    Therefore,
    $$
        \F{x+1}-\F{x-1}+2\F{y}\leq y-2+\F{2}.
    $$
    Combining this with \eqref{diff eq}, we obtain
    $$
        \F{x-1}+\F{a-2}+\F{x+1}
        \leq
        a-2+t+y-2+\F{2}
        =
        x+t+\F{2}-3.
    $$
    Thus \eqref{eq:prep-u-small} holds in this case.
\medskip

\noindent\textbf{Case 3.} 
    It remains to consider the case when
    $$
        v\leq y,\qquad v+a<v'.
    $$
    We first show that in this case one must have
    $$
        y-v\geq s.
    $$
    Suppose instead that $y-v<s$. Since $v\geq y-1$, we have
    $0\leq y-v\leq1$, and
    $$
        v+a=x+1-(y-v).
    $$
    As $y-v<s$, we have $[a,a+y-v]\subseteq I$. Since $v\in A$, the translate
    $$
        v+(A\cap[a,a+y-v])
    $$
    is contained in $[v+a,x+1]$ and is disjoint from $A$. Hence, inside the
    interval $[v+a,x+1]$, whose length is $y-v$, we have
    $$
        |A\cap[a,a+y-v]|+(\F{x+1}-\F{v+a})\leq y-v.
    $$
    On the other hand, $v+a<v'$ and $v+a\geq x$. Since
    $v'=\min(A\cap[x,\infty))$, this gives
    $$
        \F{v+a}=\F{x}.
    $$
    Using $\F{x+1}-\F{x}>1-t$, we obtain
    $$
        |A\cap[a,a+y-v]|<y-v-1+t.
    $$
    Therefore
    $$
        |A\cap[a+y-v,b]|
        >
        \frac{s+t}{2}-(y-v-1+t)
        =
        1-y+v+\frac{s-t}{2}.
    $$
    But $[a+y-v,b]$ has length $s-y+v$, and
    $$
        1-y+v+\frac{s-t}{2}\geq s-y+v
    $$
    because $s+t\leq2$. Thus
    $$
        |A\cap[a+y-v,b]|>s-y+v,
    $$
    which is impossible. Hence,   $y-v\geq s$.

    Since
    $$
        v+a=x+1-(y-v),
    $$
    and $v+a<v'$, while $v+a\geq x$, the set $A$ has no mass in $[x,v+a]$.
    Hence,
    $$
        \F{v+a}=\F{x}.
    $$
    By the previously proved gap lemma, $A$ is disjoint from $[u+b-1,x]$.
    Therefore,
    $$
        u'\leq u+b-1.
    $$
    Since $v-u\geq1$, we have $u\leq v-1$, and so
    $$
        u'\leq u+b-1\leq v+b-2=x+s-(y-v)-1.
    $$
    Since $y-v\geq s$, this gives $u'\leq x-1$. Since $x-1\in A$, we actually
    have $u'=x-1$, and hence
    $$
        \F{x}=\F{x-1}.
    $$
   Moreover, since $v\leq y$ and $v\geq y-1$, we have $0\leq y-v\leq 1$.
Thus
$
    v-y+2=2-(y-v)\in[1,2].
$
Since $\min A=1$, it follows that
$$
    \F{v-y+2}\leq 1-(y-v).
$$
    Consequently,
    $$
        \F{v+a}-\F{x-1}+\F{v-y+2}
        =
        \F{x}-\F{x-1}+\F{v-y+2}
        \leq
        1-(y-v)
        =
        v+a-x.
    $$

    By Lemma \ref{mzero}, we have $m_A^-(y)=0$, since $a+y=x+1\in A$.
    Therefore,   Lemma \ref{minterv}, applied to $[v,y]$, gives
    $$
        |(A+A)\cap[v+a,x+1]|\geq 2(\F{y}-\F{v}).
    $$
    This sumset is disjoint from $A$, and hence
    $$
        2(\F{y}-\F{v})+(\F{x+1}-\F{v+a})\leq x+1-v-a=y-v.
    $$
    Adding the last two inequalities, and using $\F{v}=\F{2}$, gives
    $$
        \F{x+1}-\F{x-1}+\F{v-y+2}+2(\F{y}-\F{2})\leq1.
    $$
    As above, since $v-u\geq1$, we have
    $$
        \F{v-y+2}\geq \F{u-(y-3)}\geq \F{u}-(y-3)=\F{2}-(y-3).
    $$
    Therefore,
    $$
        \F{x+1}-\F{x-1}+2\F{y}\leq y-2+\F{2}.
    $$
    Combining this with \eqref{diff eq}, we obtain
    $$
        \F{x-1}+\F{a-2}+\F{x+1}
        \leq x+t+\F{2}-3.
    $$
    Thus \eqref{eq:prep-u-small} holds in all cases.
\bigskip

We now finish the proof. By \eqref{w ineq} we have
    $$
        \F{a}-\F{a-2}\leq1-\F{2-w}.
    $$
    Moreover, since $a-2+w\in A$, the translate
    $$
        (a-2+w)+(A\cap[2-w,2])
    $$
    is contained in $[a,a+w]$ and is disjoint from $A$. Hence, inside
    $[a,a+1]$, the three sets
    $$
        (a-2+w)+(A\cap[2-w,2]),\qquad
        A\cap[a,b],\qquad
        A\cap[b,a+1]
    $$
    are disjoint up to endpoints. 
    Therefore
    $$
        \F{2}-\F{2-w}+\frac{s+t}{2}+(\F{a+1}-\F{b})\leq1.
    $$
    Together with \eqref{w ineq}, this gives
    $$
        \F{a}-\F{a-2}+\F{2}
        \leq
        2-\frac{s+t}{2}-(\F{a+1}-\F{b}).
    $$
    Adding this to \eqref{eq:prep-u-small} and cancelling the extra $\F{2}$,
    gives
    $$
        \F{a}+\F{x-1}+\F{x+1}
        \leq
        x-1-\frac{s-t}{2}-(\F{a+1}-\F{b}).
    $$
    Since
    $$
        \F{b}-\F{a}=\frac{s+t}{2},
    $$
    it follows that
    $$
        \F{b}+\F{x-1}+\F{x+1}
        \leq
        x-1+t-(\F{a+1}-\F{b}).
    $$
    Equivalently,
    $$
        \F{a+1}+\F{x-1}+\F{x+1}\leq x-1+t.
    $$

    Finally, since $1\in A$ and $A$ is sum free, the translate
    $$
        1+(A\cap[a,b])
    $$
    is contained in $[a+1,b+1]$ and is disjoint from $A$. Hence,
    $$
        \F{b+1}-\F{a+1}\leq\frac{s-t}{2}.
    $$
    Therefore,
    $$
        \F{x-1}+\F{b+1}+\F{x+1}
        \leq
        x-1+t+\frac{s-t}{2}
        =
        x-1+\frac{s+t}{2}
        \leq x.
    $$

    It remains only to replace $\F{b+1}$ by $\F{x}$. Since $y>3$ and $s\leq1$,
    we have
    $$
        x-b-1=y-s-2>0.
    $$
    Hence,   $b+1<x$. Also, since $u<2$, we have
    $$
        b+1>u+b-1.
    $$
    By the previously proved gap lemma, $A$ is disjoint from $[u+b-1,x]$.
    Hence,   $A$ is disjoint from $[b+1,x]$, and therefore
    $$
        \F{x}=\F{b+1}.
    $$
    We conclude that
    $$
        \F{x-1}+\F{x}+\F{x+1}\leq x,
    $$
    contradicting the choice of $x$ and ending the proof of Lemma \ref{key lemma}.

\bibliographystyle{alphaurl}
\bibliography{references}

@article{EberhardGreenManners2014,
  author       = {Sean Eberhard and Ben Green and Freddie Manners},
  title        = {Sets of integers with no large sum-free subset},
  journaltitle = {Annals of Mathematics},
  series       = {2},
  year         = {2014},
  volume       = {180},
  number       = {2},
  pages        = {621--652},
  doi          = {10.4007/annals.2014.180.2.5}
}

@article{Ruzsa1991,
  author       = {Imre Z. Ruzsa},
  title        = {Diameter of sets and measure of sumsets},
  journaltitle = {Monatshefte f{\"u}r Mathematik},
  year         = {1991},
  volume       = {112},
  number       = {4},
  pages        = {323--328},
  doi          = {10.1007/BF01351772}
}

@article{deRoton2018,
  author       = {Anne de Roton},
  title        = {Small sumsets in $\mathbb{R}$: full continuous $3k-4$ theorem, critical sets},
  journaltitle = {Journal de l'{\'E}cole polytechnique -- Math{\'e}matiques},
  year         = {2018},
  volume       = {5},
  pages        = {177--196},
  doi          = {10.5802/jep.67}
}

@article{FigalliJerison2015,
  author       = {Alessio Figalli and David Jerison},
  title        = {Quantitative stability for sumsets in $\mathbb{R}^n$},
  journaltitle = {Journal of the European Mathematical Society},
  year         = {2015},
  volume       = {17},
  number       = {5},
  pages        = {1079--1106},
  doi          = {10.4171/JEMS/527}
}

@online{GreenOpenProblems,
  author  = {Ben Green},
  title   = {100 Open Problems},
  year    = {2018},
  url     = {https://people.maths.ox.ac.uk/greenbj/papers/open-problems.pdf},
}

@article{vanHintumKeevash2026,
  author       = {Peter van Hintum and Peter Keevash},
  title        = {Locality in sumsets},
  journaltitle = {Advances in Mathematics},
  year         = {2026},
  volume       = {485},
  pages        = {110727},
  doi          = {10.1016/j.aim.2025.110727}
}

@article{Freiman1959,
  author       = {G. A. Freiman},
  title        = {The Addition of Finite Sets. I},
  journaltitle = {Izv. Vyssh. Uchebn. Zaved. Matematika},
  year         = {1959},
  number       = {6 (13)},
  pages        = {202--213},
  note         = {In Russian}
}

@misc{JingMudgal2026,
  author       = {Yifan Jing and Akshat Mudgal},
  title        = {A structure theorem for sets with doubling {$4+\delta$}},
  year         = {2026},
  eprint       = {2604.25893},
  eprinttype   = {arXiv},
  eprintclass  = {math.NT}
}

@article{Kneser1956,
  author       = {Martin Kneser},
  title        = {Summenmengen in lokalkompakten abelschen Gruppen},
  journaltitle = {Mathematische Zeitschrift},
  year         = {1956},
  volume       = {66},
  pages        = {88--110},
  doi          = {10.1007/BF01186598}
}

@article{Gardner2002,
  author       = {Richard J. Gardner},
  title        = {The {Brunn--Minkowski} inequality},
  journaltitle = {Bulletin of the American Mathematical Society},
  year         = {2002},
  volume       = {39},
  number       = {3},
  pages        = {355--405},
  doi          = {10.1090/S0273-0979-02-00941-2}
}

@incollection{Erdos1965,
  author    = {Paul Erd{\H{o}}s},
  title     = {Extremal problems in number theory},
  booktitle = {Theory of Numbers},
  series    = {Proceedings of Symposia in Pure Mathematics},
  volume    = {8},
  year      = {1965},
  pages     = {181--189},
  publisher = {American Mathematical Society},
  address   = {Providence, RI}
}

@article{Bourgain,
  author       = {Jean Bourgain},
  title        = {Estimates related to sumfree subsets of sets of integers},
  journaltitle = {Israel Journal of Mathematics},
  year         = {1997},
  volume       = {97},
  pages        = {71--92},
  doi          = {10.1007/BF02774027}
}

@online{Bedert,
  author       = {Benjamin Bedert},
  title        = {Large sum-free subsets of sets of integers via {$L^1$}-estimates for trigonometric series},
  year         = {2025},
  eprint       = {2502.08624},
  eprinttype   = {arXiv},
  eprintclass  = {math.NT}
}
\end{document}